\documentclass[11pt,twoside, a4paper,reqno]{amsart}
\pdfoutput=1

\usepackage[shorthands=off,english]{babel}
\usepackage[T1]{fontenc}
\usepackage[latin1]{inputenc}
\usepackage{amsmath,amsfonts,amssymb}
\usepackage{graphicx}
\usepackage{stmaryrd}
\usepackage{float}
\usepackage{tikz-cd}
\usepackage[all]{xy}
\usepackage[]{color}
\usepackage{xcolor}
\usepackage{textcomp}
\usepackage{latexsym}
\usepackage{xspace}
\usepackage{mathrsfs}
\usepackage[babel=false]{csquotes}
\MakeOuterQuote{"}
\usepackage{fancyvrb}
\usepackage{appendix}
\usepackage{chngcntr}
\usepackage{etoolbox}
\usepackage{lipsum}
\usepackage[formats]{listings}
\usepackage{tikz}
\DeclareMathAlphabet{\mathcal}{OMS}{cmsy}{m}{n}
\definecolor{mygreen}{rgb}{0,0.6,0}
\definecolor{mygray}{rgb}{0.5,0.5,0.5}
\definecolor{mymauve}{rgb}{0.4,0,1}
\usepackage[hidelinks,hypertexnames=false]{hyperref} 
\hypersetup{
colorlinks   = true, 
urlcolor     = blue, 
linkcolor    = blue, 
citecolor   = mymauve 
}
\usepackage{amsthm}
\usepackage[english]{cleveref}
\usepackage{enumitem}

\usepackage[margin=1in]{geometry}
\theoremstyle{plain}
\newtheorem{theorem}{Theorem}[section]{}{}

\newtheorem*{theorem*}{Theorem}
\newtheorem{lemma}[theorem]{Lemma}
\newtheorem{corollary}[theorem]{Corollary}
\newtheorem{proposition}[theorem]{Proposition}
\newtheorem{conjecture}[theorem]{Conjecture}
\newtheorem*{conjecture*}{Conjecture}

\newtheorem*{question*}{Question}

\theoremstyle{definition}

\newtheorem{definition}[theorem]{Definition}
\theoremstyle{remark}
\newtheorem{remark}[theorem]{Remark}

\crefname{theorem}{Theorem}{Theorems}
\crefname{lemma}{Lemma}{Lemmas}
\crefname{proposition}{Proposition}{Propositions}
\crefname{corollary}{Corollary}{Corollaries}
\crefname{definition}{Definition}{Definitions}
\crefname{remark}{Remark}{Remarks}
\crefname{example}{Example}{Examples}
\crefname{question}{Question}{Questions}
\crefname{conjecture}{Conjecture}{Conjectures}

\crefname{chapter}{Chapter}{Chapters}
\crefname{section}{Section}{Sections}
\crefname{figure}{Figure}{Figures}
\crefname{appendix}{Appendix}{Appendices}

\newtheoremstyle{TheoremNum}
    {8pt}{8pt}              
    {\itshape}                      
    {}                              
    {\bfseries}                     
    {.}                             
    { 5pt plus 1pt minus 1pt }                             
    {\thmname{#1}\thmnote{ \bfseries #3}}
\theoremstyle{TheoremNum}
\newtheorem{thmrep}{Theorem}
\newtheorem{cororep}{Corollary}

\newcommand{\tq}{\mathrel{{\ensuremath{\: : \: }}}}
\def\wt{\widetilde}

\let\emptyset\varnothing
\def\co{\colon}

\newcommand{\freeGroup}{F}
\newcommand{\finiteField}{\mathbb{F}}

\def\actson{\curvearrowright}

\newcommand\BigFreeProd{\displaystyle\mathop{\mbox{\huge{$\ast$}}}}

\newcommand{\Aut}{\mathrm{Aut}}
\def\PSL{\mathrm{PSL}}
\def\Sz{\mathrm{Sz}}
\def\SO{\mathrm{SO}}
\def\O{\mathrm{O}}

\def\llangle{\langle\!\langle}
\def\rrangle{\rangle\!\rangle}

\def\D{\mathbb{D}}
\def\N{\mathbb{N}}
\def\Z{\mathbb{Z}}
\def\Q{\mathbb{Q}}
\def\R{\mathbb{R}}
\def\C{\mathbb{C}}
\def\H{\mathbb{H}}

\def\FF{\mathcal{F}}

\def\SS{\mathcal{S}}

\def\K{\mathcal{K}}

\def\FF{\mathcal{F}}

\def\SLV{\mathcal{SLV}}

\def\XOS{\Gamma_{OS}}

\def\sg{\mathrm{sg}}
\def\tt{\mathbf{t}}
\def\ii{\mathbf{i}}
\def\jj{\mathbf{j}}
\def\kk{\mathbf{k}}
\def\bad{u}
\def\ttbad{\tt_{\bad}}
\def\zzbad{\mathbf{z}_{\bad}}
\def\rotMat{R}

\title{Group actions of $A_5$ on contractible $2$-complexes}
\author{Iván Sadofschi Costa}
\address{Departamento  de Matem\'atica - IMAS\\
 FCEyN, Universidad de Buenos Aires. Buenos Aires, Argentina.}
\email{isadofschi@dm.uba.ar}



\thanks{Researcher of CONICET. The author was partially supported by grants PICT-2017-2806, PIP 11220170100357CO and UBACyT 20020160100081BA}

\newcommand{\statementTheoremWi}{There is no presentation of $A_5$ of the form
$$\langle a,b,c,d,x_0,\ldots, x_k \mid a^2, b^3, c^2, d^2, (ab)^3, (bc)^2,(cd)^5, x_0ax_0^{-1}=d, w_0,\ldots, w_k\rangle$$
with $w_0,\ldots, w_k\in \ker(\phi)$, where $\phi\colon F(a,b,c,d,x_0,\ldots, x_k)\to A_5$ is given by
$a\mapsto (2,5)(3,4)$, $b\mapsto (3,5,4)$, $c\mapsto (1,2)(3,5)$, $d\mapsto (2,5)(3,4)$ and $x_i\mapsto 1$ for each $i=0,\ldots, k$.}

\newcommand{\statementTheoremCDAFive}{Every action of $A_5\cong \PSL_2(2^2)$ on a finite, contractible $2$-complex has a fixed point.}

\newcommand{\statementCharacterizationFundamentalGroupsAcyclicAFive}{Let $X$ be a fixed point free $2$-dimensional finite and acyclic $A_5$-complex and let $\pi=\pi_1(X)$. Then $\pi$ is infinite or there is an epimorphism $\pi\to A_5$.}

\newcommand{\statementCDForSimpleGroupsContainingAFive}{Let $G$ be one of the groups $\PSL_2(2^{2k})$, $\PSL_2(5^k)$ for $k\geq 1$ or $\PSL_2(q)$ for $q\equiv \pm 3 \pmod 8$ and $q\equiv \pm 1\pmod 5$.
 Then every action of $G$ on a finite contractible $2$-complex has a fixed point.}

\begin{document}
\begin{abstract}
We prove that every action of $A_5$ on a finite $2$-dimensional contractible complex has a fixed point.
\end{abstract}

\maketitle

\setcounter{tocdepth}{1}
\tableofcontents

\section{Introduction}

A well-known result of Jean-Pierre Serre \cite{SerreTrees} states that every action of a finite group on a contractible $1$-complex (i.e. a  tree) has a fixed point.
By Smith theory, every action of a $p$-group on the disk $\D^n$ has a fixed point.
The group $A_5$ acts simplicially and fixed point freely on the barycentric subdivision $X$ of the $2$-skeleton of the Poincaré homology sphere which is an acyclic $2$-complex.
By considering the join $X*A_5$, Edwin E. Floyd and Roger W. Richardson \cite{FloydRichardson} proved that $A_5$ acts simplicially and fixed point freely on a contractible $3$-complex.
Moreover, by embedding $X*A_5$ in $\R^{81}$ and taking a regular neighbourhood they proved that $A_5$ acts simplicially and fixed point freely on a triangulation of the disk $\D^{81}$.
This was the only example known of this kind until Bob Oliver obtained a complete classification of the groups that act fixed point freely on a disk $\D^n$ \cite{OliverDisks}.
The example of Floyd and Richardson makes clear that it is not possible to extend Serre's result to $3$-complexes and is natural to wonder if it holds for $2$-complexes.
Carles Casacuberta and Warren Dicks \cite{CD} made the following conjecture (without requiring $X$ to be finite)
which was also posed by Michael Aschbacher and Yoav Segev as a question \cite[Question 3]{AschbacherSegev} in the finite case.

\begin{conjecture}\label{CasacubertaDicks}
Let $G$ be a finite group. If $X$ is a $2$-dimensional finite contractible $G$-complex then $X^G\neq \emptyset$.
\end{conjecture}

In \cite{CD} the conjecture is proved for solvable groups.
Independently, Segev~\cite{Segev} studied the question of which groups act without fixed points on an acyclic $2$-complex and proved \cref{CasacubertaDicks} for solvable groups and the alternating groups $A_n$ for $n\geq 6$.
In \cite{SegevCollapsible}, Segev proved the conjecture for collapsible $2$-complexes.
Using the classification of the finite simple groups, Aschbacher and Segev proved that for many groups any action on a finite $2$-dimensional acyclic complex has a fixed point \cite{AschbacherSegev}.
Later, Oliver and Segev \cite{OS} gave a complete classification of the groups that act without fixed points on a finite acyclic $2$-complex.
Before \cite{OS}, $A_5$ was the only group known to act without fixed points on an acyclic $2$-complex.
An excellent exposition on this topic is the one given by Alejandro Adem at the Séminaire Bourbaki~\cite{Adem}.
In \cite{Corson}, J.M. Corson proved that \cref{CasacubertaDicks} holds for diagrammatically reducible complexes.
The smallest group for which \cref{CasacubertaDicks} remained open is the alternating group $A_5$.
The main result of this paper is the following.

\begin{thmrep}[\ref{theoremCDA5}]
\statementTheoremCDAFive
\end{thmrep}

From this, using the results of Oliver and Segev \cite{OS}, we deduce the following.

\begin{cororep}[\ref{CDForSimpleGroupsContainingA5}]
 \statementCDForSimpleGroupsContainingAFive
\end{cororep}

Our proof of \cref{theoremCDA5} goes by constructing a nontrivial representation in $\SO(3,\R)$ of the fundamental group of every 
fixed point free, $2$-dimensional, finite acyclic $A_5$-complex. Therefore we have the following.

\begin{cororep}[\ref{fundamentalGroupsA5}]
 \statementCharacterizationFundamentalGroupsAcyclicAFive
\end{cororep}

The paper is organized as follows.
In \cref{sectionReduction} we prove \cref{Refinement} which says that to prove \cref{theoremCDA5} it is enough to inspect the acyclic complexes of the type considered by Oliver and Segev in \cite{OS}.
The necessary results from \cite{OS} are recalled in \cref{sectionOS}.
In \cref{sectionReduction} we also prove \cref{reduccion}, which describes a possible path towards settling \cref{CasacubertaDicks}.

In \cref{sectionBrown} we establish the connection between \cref{theoremCDA5} and the following group theoretic statement, using a result of Kenneth S. Brown \cite{BrownPresentations} in Bass--Serre theory.

\begin{thmrep}[\ref{theoremWi}]
\statementTheoremWi
\end{thmrep}

In order to prove \cref{theoremWi}, in \cref{sectionModuli} we introduce a moduli of representations of the group 
$$\Gamma_k = \langle a,b,c,d,x_0, \ldots, x_k \mid a^2, b^3, c^2, d^2, (ab)^3, (bc)^2,(cd)^5, x_0 a x_0^{-1}=d \rangle$$
in $\SO(3)$.
In \cref{sectionQuaternions} we view these rotations in $S^3\subset \H$, enabling us to do a degree argument which completes the proof of \cref{theoremWi}.
This proof is inspired by James Howie's proof of the Scott--Wiegold conjecture~\cite{HowieScottWiegold}.
Finally, in \cref{sectionConsequences} we put everything together to complete the proof of \cref{theoremCDA5}.\medskip

\textbf{Note.} Some of the results presented here appeared originally in the author's thesis \cite{SadofschiCostaThesis}.

\medskip

\textbf{Acknowledgements.} 
I am grateful to my mentor, Jonathan Barmak, for his constant advice, in particular for suggesting me to work on this problem during my PhD.
I would like to thank Bob Oliver for his feedback on my thesis.
I would also like to thank Kevin Piterman for taking a look at a previous version of this article, and Yago Antolin, Ignacio Darago, Carlos Di Fiore, Fernando Martin and Pedro Tamaroff for valuable conversations.


\section[Fixed point free actions on acyclic 2-complexes]{Fixed point free actions on acyclic $2$-complexes}
\label{sectionOS}
In this section we review the results obtained by Bob Oliver and Yoav Segev in their article \cite{OS} that are needed later.

Throughout the paper, by $G$-complex we mean a $G$-CW complex.
That is, a CW complex with a continuous $G$-action that is \textit{admissible}  (i.e. the action permutes the open cells of $X$, and maps a cell to itself only via the identity).
For more details see \cite[Appendix A]{OS}.
We will frequently assume that the $2$-cells in a $G$-complex are attached along closed edge paths, this will make no difference for the questions that we study. 
A \textit{graph} is a $1$-dimensional CW complex.
By \textit{$G$-graph} we always mean a $1$-dimensional $G$-complex.

\begin{definition}[{\cite{OS}}]
A $G$-complex $X$ is \textit{essential} if there is no normal subgroup $1\neq N\triangleleft G$ such that for each $H\subseteq G$, the inclusion $X^{HN}\to X^H$ induces an isomorphism on integral homology.
\end{definition}

The main results of \cite{OS} are the following two theorems.

\begin{theorem}[{\cite[Theorem A]{OS}}]\label{teoA}
For any finite group $G$, there is an essential fixed point free $2$-dimensional (finite) acyclic $G$-complex if and only if $G$ is isomorphic to one of the simple groups $\PSL_2(2^k)$ for $k\geq 2$, $\PSL_2(q)$ for $q\equiv \pm 3 \pmod 8$ and $q\geq 5$, or $\Sz(2^k)$ for odd $k\geq 3$. Furthermore, the isotropy subgroups of any such $G$-complex are all solvable.
\end{theorem}

\begin{theorem}[{\cite[Theorem B]{OS}}]\label{teoB}
Let $G$ be any finite group, and let $X$ be any $2$-dimensional acyclic $G$-complex. Let $N$ be the subgroup generated by all normal subgroups $N'\triangleleft G$ such that $X^{N'}\neq \emptyset$. Then $X^N$ is acyclic; $X$ is essential if and only if $N=1$; and the action of $G/N$ on $X^N$ is essential.
\end{theorem}
The following fundamental result of Segev \cite[Theorem 3.4]{Segev} will be used frequently, sometimes implicitly.
We state the version given in \cite{OS}.

\begin{theorem}[{\cite[Theorem 4.1]{OS}}]\label{AciclicoOVacio}
Let $X$ be any $2$-dimensional acyclic $G$-complex (not necessarily finite). Then $X^G$ is acyclic or empty, and is acyclic if $G$ is solvable.
\end{theorem}

We denote the set of subgroups of $G$ by $\SS(G)$.

\begin{definition}[{\cite{OS}}]\label{defSeparating}
By a \textit{family} of subgroups of $G$ we mean any subset $\FF\subseteq\SS(G)$ which is closed under conjugation.
A nonempty family is said to be \textit{separating} if it has the following three properties: (a) $G\notin \FF$; (b) if $H'\subseteq H$ and $H\in \FF$ then $H'\in \FF$; (c) for any $H\triangleleft K\subseteq G$ with $K/H$ solvable, $K\in \FF$ if $H\in \FF$.

For any family $\FF$ of subgroups of $G$, a \textit{$(G,\FF)$-complex} is a $G$-complex all of whose isotropy subgroups lie in $\FF$.
A $(G,\FF)$-complex is \textit{universal} (resp. H-\textit{universal}) if the fixed point set of each $H\in \FF$ is contractible (resp. acyclic).
\end{definition}

If $G$ is not solvable, the separating family of solvable subgroups of $G$ is denoted by $\SLV$.
If $G$ is perfect, then the family of proper subgroups of $G$ is denoted by $\mathcal{MAX}$.

\begin{lemma}[{\cite[Lemma 1.2]{OS}}]\label{lemma1.2}
Let $X$ be any $2$-dimensional acyclic $G$-complex without fixed points.
Let $\FF$ be the set of subgroups $H\subseteq G$ such that $X^H\neq \emptyset$.
Then $\FF$ is a separating family of subgroups of $G$, and $X$ is an {\normalfont H}-universal $(G,\FF)$-complex.
\end{lemma}

\begin{proposition}[{\cite[Proposition 6.4]{OS}}]\label{proposition6.4}
 Assume that $L$ is one of the simple groups $\PSL_2(q)$ or $\Sz(q)$, where $q=p^k$ and $p$ is prime ($p=2$ in the second case).
 Let $G\subseteq\Aut(L)$ be any subgroup containing $L$, and let $\FF$ be a separating family for $G$.
 Then there is a $2$-dimensional acyclic $(G,\FF)$-complex if and only if $G=L$, $\FF=\SLV$, and $q$ is a power of $2$ or $q\equiv \pm 3 \pmod 8$. 
\end{proposition}

\begin{lemma}\label{lemmaNonsolvableSubgroup}
 Let $G$ be one of the groups in \cref{teoA} and let $X$ be a fixed point free $2$-dimensional acyclic $G$-complex.
 If $K\leq G$ is not solvable then the action of $K$ on $X$ is fixed point free.
 \begin{proof}
 Let $\FF=\{H\leq G \tq X^H\neq \emptyset \}$.
Then, by \cref{lemma1.2}, $\FF$ is a separating family and $X$ is an {\normalfont H}-universal $(G,\FF)$-complex.
By \cref{proposition6.4}, we must have $\FF=\mathcal{SLV}$ and therefore $X^K=\emptyset$.
 \end{proof}
\end{lemma}

If $X$ is a poset, then  $\K(X)$ denotes the \textit{order complex} of $X$, that is, the simplicial complex with simplices the finite nonempty totally ordered subsets of $X$ (the complex $\K(X)$ is also known as the \textit{nerve} of $X$).

\begin{definition}[{\cite[Definition 2.1]{OS}}]
For any family $\FF$ of subgroups of $G$ define 
$$i_\FF(H)=\frac{1}{[N_G(H):H]}(1-\chi(\K(\FF_{>H}))).$$
\end{definition}

Recall that if $G\actson X$, the orbit $G\cdot x$ is said to be \textit{of type $G/H$} if the stabilizer $G_x$ is conjugate to $H$ in $G$.
In other words, if the action of $G$ on $G\cdot x$ is the same as the action of $G$ on $G/H$.

\begin{lemma}[{\cite[Lemma 2.3]{OS}}]\label{lemma2.3}
Fix a separating family $\FF$, a finite {\normalfont H}-universal $(G,\FF)$-complex $X$, and a subgroup $H\subseteq G$.
For each $n$, let $c_n(H)$ denote the number of orbits of $n$-cells of type $G/H$ in $X$. Then $i_\FF(H)=\sum_{n\geq 0} (-1)^nc_n(H)$.
\end{lemma}

\begin{proposition}[{\cite[Tables 2,3,4]{OS}}]\label{indices}
Let $G$ be one of the simple groups $\PSL_2(2^k)$ for $k\geq 2$, $\PSL_2(q)$ for $q\equiv \pm 3 \pmod 8$ and $q\geq 5$, or $\Sz(2^k)$ for odd $k\geq 3$. Then $i_\SLV( 1 ) = 1$.
\end{proposition}

For each family of groups appearing in \cref{teoA}, Oliver and Segev describe an example.
In what follows, $D_{2m}$ is a dihedral group of order $2m$ and $C_m$ is a cyclic group of order $m$.

\begin{proposition}[{\cite[Example 3.4]{OS}}]\label{example3.4}
Set $G=\PSL_2(q)$, where $q=2^k$ and $k\geq 2$.
Then there is a $2$-dimensional acyclic fixed point free $G$-complex $X$, all of whose isotropy subgroups are solvable.
More precisely $X$ can be constructed to have three orbits of vertices with isotropy subgroups isomorphic to $B=\finiteField_q\rtimes C_{q-1}$, $D_{2(q-1)}$, and $D_{2(q+1)}$; three orbits of edges with isotropy subgroups isomorphic to $C_{q-1}$, $C_2$ and $C_2$; and one free orbit of $2$-cells. 
\end{proposition}

We have $A_5= \PSL_2(2^2)$.
The barycentric subdivision of the $2$-skeleton of the Poincar\'e dodecahedral space is an $A_5$-complex of the type given in \cref{example3.4} with fundamental group the binary icosahedral group $A_5^*\cong \mathrm{SL}(2,5)$ which has order $120$.
The Poincar\'e dodecahedral space appears in many other natural ways, for more information see \cite{KirbyScharlemann}.

\begin{proposition}[{\cite[Example 3.5]{OS}}]\label{example3.5}
Assume that $G=\PSL_2(q)$, where $q=p^k\geq 5$ and $q\equiv \pm 3 \pmod 8$. Then there is a $2$-dimensional acyclic fixed point free $G$-complex $X$, all of whose isotropy subgroups are solvable. More precisely, $X$ can be constructed to have four orbits of vertices with isotropy subgroups isomorphic to $B=\finiteField_q\rtimes C_{(q-1)/2}$, $D_{q-1}$, $D_{q+1}$, and $A_4$; four orbits of edges with isotropy subgroups isomorphic to $C_{(q-1)/2}$, $C_2^2$, $C_3$ and $C_2$; and one free orbit of $2$-cells.
\end{proposition}

\begin{proposition}[{\cite[Example 3.7]{OS}}]\label{example3.7}
Set $q=2^{2k+1}$ for any $k\geq 1$. Then there is a $2$-dimensional acyclic fixed point free $\Sz(q)$-complex $X$, all of whose isotropy subgroups are solvable. More precisely, $X$ can be constructed to have four orbits of vertices with isotropy subgroups isomorphic to $M(q,\theta)$, $D_{2(q-1)}$, $C_{q+\sqrt{2q}+1}\rtimes C_4$, $C_{q-\sqrt{2q}+1}\rtimes C_4$; four orbits of edges with isotropy subgroups isomorphic to $C_{q-1}$, $C_4$, $C_4$ and $C_2$; and one free orbit of $2$-cells.
\end{proposition}

We also have $A_5\cong\PSL_2(5)$, so this group is addressed in both \cref{example3.4} and \cref{example3.5}.
There is no other such exception.

\begin{definition}\label{defGammaOS}
If $G$ is one of the groups in \cref{teoA}, the \textit{Oliver--Segev $G$-graph $\XOS(G)$} is the $1$-skeleton of any $2$-dimensional fixed point free acyclic $G$-complex without free orbits of $1$-cells of the type constructed in \cref{example3.4,example3.5,example3.7}.
For this definition, we regard $A_5$ as $\PSL_2(2^2)$ rather than $\PSL_2(5)$.
\end{definition}

Generally, there is more than one possible choice for the $G$-graph $\XOS(G)$.  Even for $G=A_5$, thought of as $\PSL_2(2^2)$, the quotient graph $\XOS(G)/G$ is not unique.
However in \cref{UnicidadXOS} we show that $\XOS(G)$ is unique up to $G$-homotopy equivalence.
Moreover, \cref{PosiblesPi1NoDependenDeGammaOS} shows the particular choice of $\XOS(G)$ is irrelevant for our purposes.

\begin{definition}[A construction of $\XOS(A_5)$]\label{defGammaOSA5}
Here we give a construction of $\XOS(A_5)$ and we fix some notation in regard to this graph.
Consider the following subgroups of $A_5$:
\begin{align*}
H_1 & =\langle(2,5)(3,4),(3,5,4)\rangle \cong A_4,  \\
H_2 & = \langle (3,5,4),(1,2)(3,5)\rangle \cong D_{6}  \text{ and}\\
H_3 &=  \langle(1,2)(3,5),(2,5)(3,4)\rangle\cong D_{10} .
\end{align*}
The graph $\XOS(A_5)$ has three orbits of vertices whose representatives $v_1$, $v_2$, $v_3$ have stabilizers $H_1$, $H_2$, $H_3$ respectively.
In addition, $\XOS(A_5)$ has three orbits of edges whose representatives
$v_1\xrightarrow{e_{12}} v_2$,
$v_3\xrightarrow{e_{31}} v_1$ and
$v_2\xrightarrow{e_{23}} v_3$
have stabilizers 
\begin{align*}
H_{12}=H_1\cap H_2 &= \langle (3,5,4)\rangle \cong \Z_3,\\
H_{13}=H_1\cap H_3 &=\langle (2,5)(3,4)\rangle \cong \Z_2   \text{ and}\\
H_{23}=H_2\cap H_3 &=\langle (1,2)(3,5)\rangle \cong \Z_2 
\end{align*}
respectively.

Attaching a free orbit of $2$-cells to $\XOS(A_5)$ along the orbit of the closed edge path $(e_{12}, e_{23}, e_{31})$ we obtain an acyclic $2$-dimensional fixed point free $A_5$-complex of the type given in \cref{example3.4}.
This complex is, in fact, the barycentric subdivision of the $2$-skeleton of the Poincar\'e dodecahedral space  (a simplicial complex having $21=5+10+6$ vertices, $80=20+30+30$ edges and $60$ faces).
A concrete isomorphism can be produced by mapping $v_3$ to the barycentre of a pentagonal $2$-cell $ABCDE$, $v_1$ to $A$ and $v_2$ to the barycentre (midpoint) of $AB$.
For more details on this see \cite[pp. 20-21]{OS}.
\end{definition}

\section{A reduction}
\label{sectionReduction}
In this section we rely on the results of Oliver and Segev to prove \cref{Refinement}, which allows us to reduce the proof of \cref{theoremCDA5} to the study of acyclic complexes of the type considered in \cite{OS}.
We also prove \cref{reduccion} which describes a possible path to establish \cref{CasacubertaDicks}.
We first prove some results which will be used to do equivariant modifications to our complexes.

\begin{definition}
If $X,Y$ are $G$-spaces, a \textit{$G$-homotopy} is an equivariant map $H\co X\times I \to Y$.
We say that $f_0(x)=H(x,0)$ and $f_1(x)=H(x,1)$ are \textit{$G$-homotopic} and we denote this by $f_0\simeq_G f_1$.
An equivariant map $f\co X\to Y$ is a \textit{$G$-homotopy equivalence} if there is an equivariant map $g\co Y\to X$ such that $fg\simeq_G 1_Y$ and $gf\simeq_G 1_X$.
A $G$-invariant subspace $A$ of $X$ is a \textit{strong $G$-deformation retract of $X$} if there is a retraction $r\co X\to A$ such that there is a $G$-homotopy $H\co ir\simeq 1_X$ relative to $A$, where $i\co A\to X$ is the inclusion.
\end{definition}

\begin{remark}
An equivariant map $f\co X\to Y$ is a $G$-homotopy equivalence if and only if $f^H\co X^H \to Y^H$ is a homotopy equivalence for each subgroup $H\leq G$ (see \cite[Chapter II, (2.7) Proposition]{Dieck}).
Thus, if $f\co X\to Y$ is a $G$-homotopy equivalence, the action $G\actson X$ is fixed point free (resp. essential) if and only if the action $G\actson Y$ is fixed point free (resp. essential).
\end{remark}

The following lemma allows us to do elementary expansions equivariantly.
\begin{lemma}\label{LemaExpansiones}
Let $X$ be an acyclic $2$-dimensional $G$-complex. Let $H\leq G$ and $x_0,x_1\in X^{(0)}\cap X^H$.
Then there is a $G$-complex $Y\supset X$, such that $X$ is a strong $G$-deformation retract of $Y$ and $Y$ is obtained from $X$ by attaching an orbit of $1$-cells of type $G/H$ with endpoints $\{ x_0 , x_1\}$ and an orbit of $2$-cells of type $G/H$.
 
\begin{proof}
We attach an orbit of $1$-cells of type $G/H$ to $X$ using the attaching map $\varphi \co G/H \times S^0 \to X^{(0)}$ defined by  $(gH,1)\mapsto g\cdot x_0$, $(gH,-1)\mapsto g\cdot x_1$.
Let $e$ be the $1$-cell of this new orbit corresponding to the coset $H$. Since $X$ is acyclic, by \cref{AciclicoOVacio} $X^H$ is also acyclic. Let $\gamma$ be an edge path in $X^H$ starting at $x_1$ and ending at $x_0$.
Then we attach an orbit of $2$-cells of type $G/H$ in such a way that the $2$-cell corresponding to the coset $H$ is attached along the closed edge path given by $e$ and $\gamma$.
It is clear that $X$ is a strong $G$-deformation retract of $Y$.
\end{proof}
\end{lemma}

The following very natural definitions appear in \cite[Section 2]{KLV}.

\begin{definition}\label{definitionReducedGraph}
A \textit{forest} is a graph with trivial first homology.
If a subcomplex $\Gamma$ of a CW complex $X$ is a forest, there is a CW complex $Y$ obtained from $X$ by shrinking each connected component of $\Gamma$ to a point.
The quotient map $q\co X\to Y$ is a homotopy equivalence and we say $Y$ is obtained from $X$ by a \textit{forest collapse}.

If $X$ is a $G$-complex and $\Gamma\subset X$ is a forest which is $G$-invariant, the quotient map $q$ is a $G$-homotopy equivalence and we say the $G$-complex $Y$ is obtained from $X$ by a $G$-\textit{forest collapse}.
We say that a $G$-graph is \textit{reduced} if it has no edge $e$ such that $G\cdot e$ is a forest.
\end{definition}

\begin{lemma}\label{Lema1EsqueletoReducido}
Let $X$ be a $2$-dimensional acyclic $G$-complex. If $X^{(1)}$ is a reduced $G$-graph then stabilizers of different vertices are not comparable.

\begin{proof}
Let $\FF=\{ G_x \tq x\in X^{(0)}\}$ and let $M=\{ v\in X^{(0)} \tq G_v \text{ is maximal in }\FF \}$.
We first prove, by contradiction, that $X^{(0)}=M$.
Consider $v\in X^{(0)}-M$ such that $G_v$ is maximal in $\{G_{x} \tq x \in X^{(0)}-M\}$.
Then since $X^{G_v}$ contains $v$, by \cref{AciclicoOVacio} it must be acyclic.
Since $v\notin M$, there is a vertex $w\in X^{G_v}\cap M$.
By connectivity there is an edge $e\in X^{G_v}$ whose endpoints $v'$ and $w'$ satisfy $v'\notin M$ and $w'\in M$.
Since $G_{v'}\geq G_v$ and $v'\notin M$, by our choice of $v$ we have $G_v=G_{v'}$.
Since $e\in X^{G_v}$ we have $G_v\leq G_e$ and since $v'$ is an endpoint of $e$ we have $G_e \leq G_{v'}$. 
Thus $G_e=G_{v'}$ and then the degree of $v'$ in the graph $G\cdot e$ (which has vertex set $G\cdot w' \coprod G\cdot v'$) is $1$.
Thus $G\cdot e$ is a forest, contradiction.
Therefore we must have $M=X^{(0)}$.
To conclude we have to prove that different vertices $u,v\in M$ have different stabilizers.
Suppose $G_u=G_{v}$ to get a contradiction.
Since $u,v$ are vertices of $X^{G_u}$ which is connected, there is an edge $e\in X^{G_u}$ and by maximality we must have $G_e=G_u$.
If $u',v'$ are the endpoints of $e$, we have $G_{u'}=G_{v'}$. We have two cases and in any case we obtain a contradiction.
If $G\cdot u' \neq G \cdot v'$ then $G\cdot e$ is a forest consisting of $|G/G_e|$ disjoint edges, contradiction.
Otherwise, there is a nontrivial element $g\in G$ such that $g\cdot u' = v'$ and we have $G_{u'}= G_{v'} = gG_{u'}g^{-1}$. Thus $g\in N_G( G_{u'} )$.
Consider the action of $\langle g \rangle$ on $X^{G_{u'}}$, which is acyclic and thus has a fixed point by the Lefschetz fixed point theorem. But this cannot happen, since this would imply that $\langle G_{u'}, g \rangle \gneq G_{u'}$ fixes a point of $X$, which is a contradiction since $u'\in M$.
\end{proof}
\end{lemma}

Now we prove the main results of the section.

\begin{theorem}\label{Refinement}
 Let $G$ be one of the groups in \cref{teoA}.
 Let $X$ be a fixed point free $2$-dimensional finite acyclic $G$-complex.
 Then there is a fixed point free $2$-dimensional finite acyclic $G$-complex $X'$ obtained from the $G$-graph $\XOS(G)$ by attaching $k\geq 0$ free orbits of $1$-cells and $k+1$ free orbits of $2$-cells and an epimorphism $\pi_1(X)\to \pi_1(X')$.
  \begin{proof}
    Let $\FF=\{H\leq G \tq X^H\neq \emptyset \}$.
    Then, by \cref{lemma1.2}, $\FF$ is a separating family and $X$ is an {\normalfont H}-universal $(G,\FF)$-complex.
    By \cref{proposition6.4}, we must have $\FF=\mathcal{SLV}$.
    By doing enough $G$-forest collapses we can assume that $X^{(1)}$ is a reduced $G$-graph.
    The stabilizers of the vertices of $\XOS(G)$ are precisely the maximal solvable subgroups of $G$.
    Therefore, since every solvable subgroup of $G$ fixes a point of $X$,
    by \cref{Lema1EsqueletoReducido}, we may identify  $X^{(0)}=\XOS(G)^{(0)}$.
    Applying \cref{LemaExpansiones} enough times to modify $X$, we may further assume $\XOS(G)$ is a subcomplex of $X$.

    Finally we will modify $X$ so that for every subgroup $1\neq H\leq G$, we have $X^H=\XOS(G)^H$.
    We do this by reverse induction on $|H|$.
    Assume that we have $X$ such that it holds for every subgroup $K$ with $H\lneq K\leq G$.
    If $H$ is not solvable, we have $X^H=\XOS(G)^H=\emptyset$ so we are done.
    If $H$ is solvable, since $\XOS(G)^H$ is a tree (it is acyclic and $1$-dimensional) and $X^H$ is acyclic by \cref{AciclicoOVacio}, the inclusion $\XOS(G)^H\hookrightarrow X^H$ is an $N_G(H)$-equivariant homology equivalence. 
    Now since $\XOS(G)^H$ is a tree we can define an $N_G(H)$-equivariant retraction $r_H\co  X^H\to \XOS(G)^H$.
    Then $r_H$ is a homology equivalence.
    Moreover, the stabilizer of the cells in $X^H - \XOS(G)^H$ is $H$ (the stabilizer cannot be bigger by the induction hypothesis).
    We define retractions $r_{H^g}\co X^{H^g}\to \XOS(G)^{H^g}$ by  $r_{H^g}(gx)=g\cdot r_H(x)$ which glue to give a $G$-equivariant homology equivalence
      $$r\co \XOS(G) \bigcup_{g\in G} X^{H^g}\to \XOS(G).$$
    We may replace $X$ by the pushout $\wt{X}$ given by the following diagram
    \begin{center}
    \begin{tikzcd}
    \XOS(G) \displaystyle\bigcup_{g\in G} X^{H^g} \arrow[hook]{d}[]{}\arrow{r}[]{r} & \XOS(G)\arrow[hook]{d}{} \\
    X\arrow[]{r}[swap]{\overline{r}}& \wt{X}
    \end{tikzcd}
    \end{center}
    It follows that $\overline{r}$ is a homology equivalence, so the resulting $G$-complex $\wt{X}$ is acyclic.
    Moreover since $\wt{X}^{(1)}$ is a subcomplex of $X^{(1)}$ and the restriction $\overline{r}\colon X^{(1)}\to \wt{X}^{(1)}$ is a retraction,
    $\overline{r}$ induces an epimorphism on $\pi_1$.
    This procedure removes the excessive orbits of cells of type $G/H$.
    By induction we obtain a complex $X'$ such that $X'^{(1)}$ coincides with $\XOS(G)$ up to $k\geq 0$ free orbits of $1$-cells and such that every orbit of $2$-cells is free.
    By \cref{lemma1.2} $X'$ is an H-universal $(G,\mathcal{SLV})$-complex.
    Now by \cref{lemma2.3} and \cref{indices} there are exactly $k+1$ orbits of $2$-cells.
\end{proof}
\end{theorem}

From \cite[Propositions 3.3 and 3.6]{OS} we have:

\begin{proposition}\label{minimalSimpleGroupsOfOliverSegev}
Each of the groups in the statement of \cref{teoA} has a subgroup isomorphic to one of the following groups:
\begin{itemize}
\item $\PSL_2(2^p)$ for $p$ prime;

\item $\PSL_2(3^p)$ for an odd prime $p$;

\item $\PSL_2(q)$ for a prime $q>3$ such that  $q\equiv \pm 3 \pmod 5$ and $q\equiv \pm 3 \pmod 8$;

\item $\Sz(2^p)$ for $p$ an odd prime.
\end{itemize}
Moreover, every proper subgroup of a group in this list is solvable.
\end{proposition}

\begin{theorem}\label{reduccion}
To prove \cref{CasacubertaDicks} it is enough to prove $P(G)$ 
for each group $G$ listed in \cref{minimalSimpleGroupsOfOliverSegev},
where $P(G)$ denotes the following proposition:
``there is a nontrivial representation in $\SO(n,\R)$
of the fundamental group of every acyclic $G$-complex obtained from $\XOS(G)$ by attaching $k\geq 0$ free orbits of $1$-cells and $(k+1)$ free orbits of $2$-cells''.

\begin{proof}
Let $G$ be a finite group and suppose that $X$ is a finite, acyclic $2$-dimensional fixed point free $G$-complex.
Let $N$ be the subgroup generated by all normal subgroups $N'\triangleleft G$ such that $X^{N'}\neq \emptyset$.
By \cref{teoB} we have that $Y=X^N$ is acyclic and the action of $G/N$ on $Y$ is essential and fixed point free.
Then $G/N$ must be one of the groups in \cref{teoA}.
We take a subgroup $K$ of $G/N$ isomorphic to one of the groups listed in \cref{minimalSimpleGroupsOfOliverSegev}.
Then by \cref{lemmaNonsolvableSubgroup} the action of $K$ on $Y$ is fixed point free.
Now, by \cref{Refinement} and by $P(K)$, it follows that $\pi_1(Y)$ admits a nontrivial representation in $\SO(n,\R)$. 
Therefore, by \cref{casoParticularConjeturaSubcomplejo}, $X$ cannot be contractible.
\end{proof}
\end{theorem}

\begin{remark}
In \cref{A5ComplexHasRepInSO3} we prove the group $A_5\cong \PSL_2(2^2)$ satisfies the condition $P$ in \cref{reduccion}.
\end{remark}

The following explains why our particular choice of $\XOS(G)$ and the way the free orbits of $1$-cells are attached is not relevant.
\begin{proposition}\label{UnicidadXOS}
Any two choices for $\XOS(G)$ are $G$-homotopy equivalent. Moreover, attaching $k\geq 0$ free orbits of $1$-cells to any two choices for $\XOS(G)$ produces $G$-homotopy equivalent graphs.
\begin{proof}
Since any choice of $\XOS(G)$ is a universal $(G,\mathcal{SLV}-\{1\})$-complex, the first part follows from \cite[Proposition A.6]{OS}.
The second part follows easily from the first and the gluing theorem for adjunction spaces {\cite[7.5.7]{TopologyGroupoids}}.
\end{proof}
\end{proposition}

\begin{corollary}\label{PosiblesPi1NoDependenDeGammaOS}
Let $\Gamma$ be a graph obtained from $\XOS(G)$ by attaching $k\geq 0$ free orbits of $1$-cells.
The set of $G$-homotopy equivalence classes of $2$-dimensional acyclic fixed point free $G$-complexes with $1$-skeleton $\Gamma$ does not depend on the particular choice of $\XOS(G)$ or the way the $k$ free orbits of $1$-cells are attached.
In particular, the set of isomorphism classes of groups that occur as the fundamental group of such spaces does not depend on such choices.
\begin{proof}
Again, this is an easy application of {\cite[7.5.7]{TopologyGroupoids}}.
\end{proof}
\end{corollary}

\section{Brown's short exact sequence}
\label{sectionBrown}
Using Bass--Serre theory, K.S. Brown gave a method to produce a presentation for a group $G$ acting on a simply connected complex $X$ \cite[Theorem 1]{BrownPresentations}.
When $X$ is not simply connected, Brown describes a presentation for an extension $\wt{G}_X$ of $G$ by $\pi_1(X)$  \cite[Theorem 2]{BrownPresentations}.
The group $\wt{G}_X$ has a description as a quotient of the fundamental group of a graph of groups.
A similar result in the simply connected case was given by Corson \cite[Theorem 5.1]{CorsonComplexesOfGroups} in terms of \textit{complexes of groups} (higher dimensional analogues of graphs of groups).

Using Brown's result we translate the $A_5$ case of \cref{CasacubertaDicks} into a nice looking problem in combinatorial group theory.
This translation can be done in general, but to obtain similar results for the rest of the groups $G$ that appear in \cref{reduccion} we need a choice of $\XOS(G)$ and presentations for the stabilizers of its vertices.

In Brown's original formulation, the result deals with actions that need not to be admissible (Brown uses the term $G-CW$-complex in a different way than us).
Since the actions we are interested in are admissible, we state Brown's result only in that case.

Let $X$ be a connected $G$-complex. By admissibility of the action, the group $G$ acts on the set of oriented edges.
If $e$ is an oriented edge, the same $1$-cell with the opposite orientation is denoted by ${e}^{-1}$.
Each oriented edge $e$ has a \textit{source} and \textit{target}, denoted by $s(e)$ and $t(e)$ and for every $g\in G$ we have $g\cdot s(e)= s(g\cdot e)$ and $g\cdot t(e)= t(g\cdot e)$.

To obtain a description of the group $\wt{G}_X$ we need a number of choices that we now specify.
For each $1$-cell of $X$ we choose a preferred orientation in such a way that these orientations are preserved by $G$.
This determines a set $P$ of oriented edges.
We choose a \textit{tree of representatives} for $X/ G$.
That is, a tree $T\subset X$ such that the vertex set $V$ of $T$ is a set of representatives of $X^{(0)}/G$.
Such tree always exists and the $1$-cells of $T$ are inequivalent modulo $G$. We give an orientation to the $1$-cells of $T$ so that they are elements of $P$.
We also choose a set of representatives $E$ of $P/G$ in such a way that $s(e)\in V$ for every $e\in E$ and such that each oriented edge of $T$ is in $E$.
If $e$ is an oriented edge, the unique element of $V$ that is equivalent to $t(e)$ modulo $G$ will be denoted by $w(e)$.
For every $e\in E$ we fix an element $g_e\in G$ such that $t(e)=g_e\cdot w(e)$.
If $e\in T$, we specifically choose $g_e=1$.
For each orbit of $2$-cells we choose a closed edge path $\tau$ based at a vertex of $T$ and representing the attaching map for this orbit of $2$-cells.
Let $F$ be the set given by these closed edge paths.

The group $\wt{G}_X$ is defined as a quotient of $$\BigFreeProd_{v\in V} G_v * \BigFreeProd_{e \in E} \Z$$ by 
certain relations.
In order to define these relations we introduce some notation.
If $v\in V$ and $g\in G_v$ we denote the copy of $g$ in the free factor $G_v$ by $g_v$.
The generator of the copy of $\Z$ that corresponds to $e$ is denoted by $x_e$.
The relations are the following:

(i) $x_e=1$ if $e\in T$.

(ii) $x_e^{-1} g_{s(e)} x_e = (g_e^{-1} g g_e)_{w(e)}$ for every $e\in E$ and $g\in G_e$.

(iii) $r_\tau=1$ for every $\tau\in F$.

We state Brown's theorem before giving the definition of the element $r_\omega$ associated to a closed edge path $\omega$.

\begin{theorem}[{Brown, \cite[{Theorems 1 and 2}]{BrownPresentations}}]\label{resultadoBrown}
The group
$$\wt{G}_X= \frac{\BigFreeProd_{v\in V}G_v * \BigFreeProd_{e \in E} \Z  }{\llangle R \rrangle}$$ where $R$ consists of relations (i)-(iii)
is an extension
$$1\to \pi_1(X,x_0)\xrightarrow{i} \wt{G}_X \xrightarrow{\overline{\phi}} G\to 1.$$
The map $\overline{\phi}$ is defined passing to the quotient the coproduct $\phi$ of the inclusions $G_v\to G$ and the mappings $\Z\to G$ given by $x_e\mapsto g_e$.
The map $i$ sends a closed edge path $\omega$ based at $x_0\in V$ to the class of $r_\omega$.
\end{theorem}

Now we explain how to obtain the elements $r_\omega$. 
If $\alpha$ is an oriented edge, we define 
$$\varepsilon(\alpha)=\begin{cases} 
\hspace{9pt} 1 & \text{$\alpha\in P$} \\
-1 & \text{if $\alpha\notin P$}
                       \end{cases}$$
and we can always take $e\in E$ and $g\in G$ such that $\alpha=g e^{\varepsilon(\alpha)}$. Note that $e$ is unique but $g$ is not. Moreover, if $\alpha$ starts at $v\in V$, we can write
$$
\alpha = \begin{cases}
          h e       & \text{ with $h\in G_{s(e)}$, if $\alpha\in P$} \\
           h g_e^{-1} e^{-1} & \text{ with $h\in G_{w(e)}$, if $\alpha\notin P$} \\
         \end{cases}
$$
Again, $h$ is not unique.

Now if $\omega = (\alpha_1,\ldots, \alpha_n)$ is a closed edge path starting at a vertex $v_0\in V$ we define the element $r_\omega\in \BigFreeProd_{v\in V}G_v * \BigFreeProd_{e \in E} \Z$.
Recursively, we define some sequences.
Since the oriented edge $\alpha_1$ starts at $v_0\in V$, we can obtain an oriented edge $e_1$ and an element $h_1\in G_{v_0}$ as above.
We set $\varepsilon_1=\varepsilon(\alpha_1)$ and $g_1=h_1 g_{e_1}^{\varepsilon_1}$.
Set $v_1=w(e_1)$ if $\alpha_1\in P$ and otherwise $v_1=s(e_1)$.
Now suppose we have defined $e_1,\ldots, e_k$, $h_1,\ldots, h_k$, $\varepsilon_1,\ldots, \varepsilon_k$, $g_1,\ldots,g_k$ and $v_1,\ldots, v_k$ such that the oriented edge $(g_1g_2\cdots g_k)^{-1}\alpha_{k+1}$ starts at $v_k\in V$. We can obtain an oriented edge $e_{k+1}$ and an element $h_{k+1}\in G_{v_{k}}$ as before.
We set $\varepsilon_{k+1}=\varepsilon(\alpha_{k+1})$ and $g_{k+1}=h_{k+1} g_{e_{k+1}}^{\varepsilon_{k+1}}$.
Set $v_{k+1}=w(e_{k+1})$ if $\alpha_{k+1}\in P$ and otherwise $v_{k+1}=s(e_{k+1})$.
When we conclude, we have an element $g_1g_2\cdots g_n\in G_{v_0}$.
Finally the word associated to $\omega$ is given by
$$r_\omega \, = \, {(h_1)}_{v_0} x_{e_1}^{\varepsilon_1}\,\,  {(h_2)}_{v_1} x_{e_2}^{\varepsilon_2}\,\, \cdots \,\,{(h_n)}_{v_{n-1}} x_{e_n}^{\varepsilon_n} \,\,(g_1g_2\cdots g_n)_{v_0}^{-1}.$$

A closed edge path $\omega$ in $X$ determines a conjugacy class $\llbracket\omega\rrbracket$ of $\pi_1(X)$.
The following describes the conjugation action of $\wt{G}_X$ on $\pi_1(X)$.
\begin{proposition}[{\cite[Proposition 1]{BrownPresentations}}]\label{AccionBrown}
Let $\omega$ be a closed edge path in $X$ and $g\in G$.
Then the conjugacy classes $i(\llbracket\omega\rrbracket)$ and $i(\llbracket g\omega\rrbracket)$ are contained in the same $\wt{G}_X$-conjugacy class. 
Moreover for any element $\wt{g}\in \overline{\phi}^{-1}(g)$ we have 
$\wt{g}i( \llbracket \omega\rrbracket) \wt{g}^{-1} = i( \llbracket g\omega\rrbracket)$.
\end{proposition}

The following proposition summarizes many ideas of this section.

\begin{proposition}
\label{prop3by3}
Let $\Gamma$ be a $G$-graph and let $w_1,\ldots,w_k\in \ker( \overline{\phi}\co \wt{G}_\Gamma\to G)$.
Let $X$ be a $G$-complex obtained by attaching orbits of $2$-cells to $\Gamma$ along closed edge paths $\tau_1,\ldots,\tau_k$ such that $r_{\tau_i}=w_i$.
Then we have a diagram with exact rows and columns

\begin{center}
\begin{tikzcd}
  & 1\arrow{d}{} & 1\arrow{d}{} & 1\arrow{d}{} & \\
1\arrow{r}{} & \llangle G\cdot \tau_i \rrangle \arrow[hook]{d}{}\arrow[]{r}{i}[swap]{\sim} & \llangle w_i \rrangle^{\wt{G}_\Gamma}\arrow[hook]{d}{}\arrow{r}{} & 1\arrow{r}{}\arrow{d}{} & 1 \\
1 \arrow{r}{}& \pi_1(\Gamma)\arrow{d}{i_*}\arrow{r}{i} & \wt{G}_\Gamma\arrow{d}{}\arrow{r}{\overline{\phi}} & G \arrow[equals]{d}{}\arrow{r}{}& 1 \\
1 \arrow{r}{} & \pi_1( X ) \arrow{r}{i}\arrow{d}{}& \wt{G}_X \arrow{r}{\overline{\phi}}\arrow{d}{}& G\arrow{r}{}\arrow{d}{} & 1 \\
  & 1 & 1 & 1 &
\end{tikzcd}
\end{center}
and we have
$\displaystyle H_1(X)\cong \frac{N}{\llangle w_i \rrangle^{\wt{G}_\Gamma}[N,N] }$,
where $N=\ker(\overline{\phi}\colon \wt{G}_\Gamma\to G)$.
\end{proposition}

\begin{remark}
If $X$ is a connected $G$-complex, the group $\wt{G}_X$ is isomorphic to the group formed by the pairs $(g,\wt{g})$ such that $g\in G$ and $\wt{g}$ is a lift of $g\co X\to X$ to the universal cover $\wt{X}$ of $X$ (see \cite{BrownPresentations}).
Suppose $Y$ is another $G$-complex and $h\co X\to Y$ is equivariant and a homotopy equivalence.
Let $\wt{h}\co \wt{X}\to \wt{Y}$ be a lift of $h$ to the universal covers.
Then if $g\in G$, for each lift $\wt{g}_X: \wt{X}\to \wt{X}$  of $g\co X\to X$ there is a unique lift $\wt{g}_Y\co \wt{Y}\to \wt{Y}$ of $g\co Y\to Y$ such that the following diagram commutes:
\begin{center}
\begin{tikzcd}
\wt{X}\arrow[]{r}[]{\wt{h}} \arrow[]{d}[swap]{\wt{g}_X} & \wt{Y} \arrow[]{d}[]{\wt{g}_Y} \\
\wt{X} \arrow[]{r}[swap]{\wt{h}} & \wt{Y}
\end{tikzcd}
\end{center}
Then it is easy to check that there is an isomorphism $\wt{G}_X\to \wt{G}_Y$ given by $\wt{g}_X\mapsto \wt{g}_Y$.
In particular, the isomorphism type of $\wt{G}_{\XOS(G)}$ does not depend on any choice.
\end{remark}


We now apply Brown's result for $G=A_5$.
Recall the construction of $\XOS(A_5)$ given in \cref{defGammaOSA5}.
Suppose that we have an acyclic $2$-complex $X$ obtained from $\XOS(A_5)$ by attaching a free $A_5$-orbit of $2$-cells.
We want to apply Brown's method to obtain a presentation for the extension $\wt{G}_X$.
We take $T=\{ e_{12}, e_{23}\}$.
Thus $V=\{v_1,v_2,v_3\}$.
We take $E=\{e_{12}, e_{23}, e_{31}\}$.
Note that we have $w(e)=t(e)$ for every $e\in E$.
We can take $g_e=1$ for every $e\in E$.

Then Brown's result gives
$$\wt{G}_X = \frac{(H_1*_{H_{12}} H_2 *_{H_{23}} H_3 )*_{H_{13}}}{\llangle w\rrangle}$$
We explain this.
First we amalgamate the groups $H_1$, $H_2$, $H_3$ identifying the copy of $H_{12}$ in $H_1$ with the copy of $H_{12}$ in $H_2$ and the copy of $H_{23}$ in $H_2$ with the copy of $H_{23}$ in $H_3$.
This comes from the relations of type (i) and (ii) for $e\in T$.
Then we form an HNN extension with stable letter $x=x_{e_{31}}$ that corresponds to the relation of type (ii) coming from $e_{31}$.
The associated subgroups of this HNN extension are the copies of $H_{13}$ in $H_1$ and $H_3$.
The  quotient by the word $w$ comes from the only relation of type (iii).

Now we obtain an explicit presentation for $\wt{G}_X$.
We have $\langle a,b \mid a^2, b^3, (ab)^3\rangle\cong H_1$ via $a\mapsto (2,5)(3,4)$, $b\mapsto (3,5,4)$.
We have $\langle b,c \mid b^3,c^2, (bc)^2\rangle\cong H_2$ via $b\mapsto (3,5,4)$, $c\mapsto (1,2)(3,5)$.
Finally $\langle c, d \mid c^2, d^2 , (cd)^5\rangle\cong H_3$ via $c\mapsto (1,2)(3,5)$, $d\mapsto (2,5)(3,4)$.
Thus we have a presentation
$$\wt{G}_X=\langle a,b,c,d,x\mid a^2, b^3, c^2, d^2, (ab)^3, (bc)^2,(cd)^5, xax^{-1}=d, w\rangle$$
where the word $w$ depends on the attaching map.
The mapping $\overline{\phi}\co \wt{G}_X\to A_5$ is given by $a\mapsto (2,5)(3,4)$, $b\mapsto (3,5,4)$, $c\mapsto (1,2)(3,5)$, $d\mapsto (2,5)(3,4)$ and $x\mapsto 1$.
Note that $\pi_1(X)$ is trivial if and only if $\overline{\phi}\co \wt{G}_X\to A_5$ is an isomorphism.
If we also take into account $k$ additional free orbits of $1$ and $2$ cells and we recall \cref{Refinement}, from Brown's result we obtain:

\begin{theorem}\label{TheoremEquivalenceA5CDAndWi}
The following are equivalent.
\begin{enumerate}
\item[(i)] Every finite, $2$-dimensional contractible $A_5$-complex has a fixed point.

\item[(ii)] There is no presentation of $A_5$ of the form
$$\langle a,b,c,d,x_0, \ldots, x_k \mid a^2, b^3, c^2, d^2, (ab)^3, (bc)^2,(cd)^5, x_0ax_0^{-1}=d, w_0, \ldots, w_k \rangle$$
with $w_0,\ldots, w_k\in \ker(\phi)$, where  $\phi\co  F(a,b,c,d,x_0,\ldots,x_k)\to A_5$ is given by $a\mapsto (2,5)(3,4)$, $b\mapsto (3,5,4)$, $c\mapsto (1,2)(3,5)$, $d\mapsto (2,5)(3,4)$ and $x_i\mapsto 1$.
\end{enumerate}

\end{theorem}

\section{A moduli of representations}
\label{sectionModuli}
In order to prove \cref{theoremWi} we define a moduli of representations of the group 
$$\Gamma_k = \langle a,b,c,d,x_0, \ldots, x_k \mid a^2, b^3, c^2, d^2, (ab)^3, (bc)^2,(cd)^5, x_0 a x_0^{-1}=d \rangle$$ in $\SO(3)$.
Our argument is inspired by James Howie's proof of the Scott--Wiegold conjecture~\cite{HowieScottWiegold}.

Let $\overline{\phi}\colon \Gamma_k \to A_5$ be the homomorphism induced by $\phi\colon F(a,b,c,d,x_0,\ldots, x_k)\to A_5$.

\begin{lemma}\label{lemaXandBAC3}
We have $\ker\left(\overline{\phi}\right) = \llangle x_0,\ldots,x_k, (bac)^3\rrangle$.
\begin{proof}
It is straightforward to verify that $(bac)^3\in \ker(\overline{\phi})$, so it is enough to show the induced epimorphism
$$\overline{\overline{\phi}}\colon \Gamma_k/\llangle x_0,\ldots, x_k, (bac)^3\rrangle\to A_5$$
is in fact an isomorphism.
Eliminating $d$ and the $x_i$ we see
$$\Gamma_k/\llangle x_0,\ldots, x_k, (bac)^3\rrangle = \langle a,b,c \mid a^2,b^3, c^2, (ab)^3, (bc)^2, (ca)^5, (bac)^3 \rangle.$$
The quickest way to finish the proof is by using \textsf{GAP} \cite{GAP} to compute the order of this group:
\begin{Verbatim}
gap> F:=FreeGroup("a","b","c");;
gap> AssignGeneratorVariables(F);;
#I  Assigned the global variables [ a, b, c ]
gap> R:=[a^2,b^3,c^2,(a*b)^3,(b*c)^2,(c*a)^5,(b*a*c)^3];;
gap> Order(F/R);
60
\end{Verbatim}
In \cref{appendixGAP} we give an alternative proof by hand.
\end{proof}
\end{lemma}

\begin{proposition}\label{goodRep}
Let $w_0,\ldots,w_k \in \ker(\phi)$. If the group $\Gamma_k$ admits a representation $\rho$ such that 

(i) $\rho(w_i)=1$ for each $i=0,\ldots,k$ and

(ii) there exists $r\in \{x_0,\ldots,x_k, (bac)^3\}$ such that  $\rho(r)\neq 1$

then $\Gamma_k/\llangle w_0,\ldots, w_k \rrangle\xrightarrow{\overline{\overline{\phi}}} A_5$ is not an isomorphism.
\begin{proof}
This follows from \cref{lemaXandBAC3}.
\end{proof} 
\end{proposition}

\begin{remark}
Note that in some cases (for example when $k=0$ and $w_0=x_0$) a representation of $\Gamma_k$ with image isomorphic to $A_5$ may suffice to conclude that $\Gamma_k/\llangle w_0,\ldots, w_k \rrangle$ is not $A_5$.
This may seem counterintuitive.
\end{remark}

If $\alpha,\beta\in \C$ we consider the matrix
$$\rotMat(\alpha,\beta)= 
\begin{pmatrix}
\alpha & \beta & 0 \\
-\beta & \alpha & 0\\
0 & 0 & 1 \\
\end{pmatrix}$$
which lies in $\SO(3,\C)$ whenever $\alpha^2+\beta^2=1$.
Recall that $\SO(n,\C)$ is the group of matrices $M\in M_n(\C)$ such that $M\cdot M^T=1$ and $\det(M)=1$.
We now introduce our moduli of representations of $\Gamma_k$.
\begin{theorem}\label{teoModuli}
If $\mathbf{z}=(\alpha_1,\beta_1,\alpha_2,\beta_2,\alpha_3,\beta_3,X_1,\ldots, X_k)\in\C^6 \times \SO(3,\C)^k$
satisfies $\alpha_i^2+\beta_i^2=1$ for $i=1,2,3$
then there is a group representation
$$\rho_{\mathbf{z}}\colon \Gamma_k \to \SO(3,\C)$$
defined by the following matrices

\begin{align*}
A &= \begin{pmatrix}
-1 & 0 & 0 \\
0 & \frac{1}{3} & -\frac{2}{3}\sqrt{2}\\
0 & -\frac{2}{3}\sqrt{2} & -\frac{1}{3} \\
\end{pmatrix}\\
B &= \begin{pmatrix}
-\frac{1}{2} &-\frac{\sqrt{3}}{2} & 0  \\
\frac{\sqrt{3}}{2} & -\frac{1}{2} & 0  \\
0 & 0 & 1
\end{pmatrix}\\
C &= \rotMat(\alpha_1,\beta_1)\, S_0\,\rotMat(\alpha_1,\beta_1)^T \\
D &= \rotMat(\alpha_1,\beta_1)\, S_1\, \rotMat(\alpha_2,\beta_2)\, S_2\, \rotMat(\alpha_2,\beta_2)^T\, S_1^T\, \rotMat(\alpha_1,\beta_1)^T\\
X_0 &= \rotMat(\alpha_1,\beta_1)\, S_1\, \rotMat(\alpha_2,\beta_2) \, S_3 \, \rotMat(\alpha_3,\beta_3)\, S_4
\end{align*}
where 

$S_0 = \begin{pmatrix}
-1 & 0 & 0 \\
0 & 1 & 0 \\
0 & 0 & -1 
\end{pmatrix}$,
$S_1=\begin{pmatrix}
-1 & 0 & 0\\
0 & 0 & -1\\
0 & -1 & 0
\end{pmatrix}
$,
$S_2=\begin{pmatrix}
 - \cos(\frac{2\pi}{5})& 0 & -\sin(\frac{2\pi}{5})  \\
 0&-1 &0 \\
 -\sin(\frac{2\pi}{5}) & 0 & \cos(\frac{2\pi}{5})
\end{pmatrix}
$,\medskip

$S_3=\begin{pmatrix}
 0 & \cos(\frac{\pi}{5}) & \sin(\frac{\pi}{5})  \\
 1 & 0 & 0 \\
 0 & \sin(\frac{\pi}{5}) & -\cos(\frac{\pi}{5}) 
\end{pmatrix}
$
and
$S_4=\begin{pmatrix}
0 & -\frac{\sqrt{3}}{3} & -\frac{\sqrt{6}}{3} \\
1 &0 &0 \\
0 & -\frac{\sqrt{6}}{3} & \frac{\sqrt{3}}{3}
\end{pmatrix}$.

\begin{proof}
The proof reduces to the case $k=0$.
We describe the computations needed to finish the proof.
It is straightforward to prove $A^2=1$, $B^3=1$ and $(AB)^3=1$; $C^2=1$ and $D^2=1$ reduce to $S_0^2=1$ and $S_2^2=1$ respectively.
Since $\rotMat(\alpha_1,\beta_1)$ commutes with $B$, to prove $(BC)^2=1$ it is enough to verify $(BS_0)^2=1$.
To prove $(CD)^5=1$ it is enough to verify that $(S_1^TS_0S_1\rotMat(\alpha_2,\beta_2)S_2\rotMat(\alpha_2,\beta_2)^T)^5=1$ and, since
$$S_1^TS_0S_1= \begin{pmatrix} -1 & 0 & 0 \\ 0 & -1 & 0 \\ 0 & 0 & 1\end{pmatrix}$$
commutes with $\rotMat(\alpha_2,\beta_2)$, this reduces to proving $(S_1^TS_0S_1S_2)^5=1$ which follows from
$$S_1^TS_0S_1S_2 = \begin{pmatrix}
\cos(\frac{2\pi}{5}) & 0 & \sin(\frac{2\pi}{5}) \\
0 & 1 & 0 \\
-\sin(\frac{2\pi}{5}) & 0 & \cos(\frac{2\pi}{5})
\end{pmatrix}.$$
Finally, $X_0 A X_0^T = D$ reduces to $S_3 \rotMat(\alpha_3,\beta_3)S_4 A S_4^T \rotMat(\alpha_3,\beta_3)^T S_3^T = S_2$ which follows from
$$S_4AS_4^T = S_3^TS_2S_3 = \begin{pmatrix} -1 & 0 & 0 \\ 0 & -1 & 0 \\ 0 & 0 & 1\end{pmatrix}.$$
The function \verb|check_rep| in \cref{appendixSAGE} gives an alternative proof using SAGE \cite{SAGE}. 
\end{proof}
\end{theorem}

\begin{remark}
We also regard $A$, $B$, $C$, $D$, $X_0$ as matrices with coefficients in the polynomial ring $\C[\alpha_1,\beta_1,\alpha_2,\beta_2,\alpha_3,\beta_3]$. 
\end{remark}

\begin{remark}
 This family of representations was obtained in the following way.
 We first obtained a single representation of the group $\Gamma_0$ in $\SO(3,\R)$ by choosing reflections $\sigma_1,\sigma_2,\sigma_3,\sigma_4,\sigma_5$
 with axes forming the appropriate angles so that $a\mapsto \sigma_1\sigma_2$, $b\mapsto \sigma_2\sigma_3$, $c\mapsto \sigma_3\sigma_4$ and $d\mapsto \sigma_4\sigma_5$ defines a representation of the (alternating Coxeter) group generated by $a$, $b$, $c$, and $d$.
 Since $\sigma_1\sigma_2$ and $\sigma_4\sigma_5$ are rotations of the same angle, they are conjugate, so it is possible to extend this to a representation of $\Gamma_0$ by mapping $x_0$ to a rotation $r$.
 Then we \textit{twisted} this representation in the following way to obtain three degrees of freedom.
 If $\theta_1$, $\theta_2$ and $\theta_3$ are rotations commuting with $\sigma_1\sigma_2$, $\sigma_2\sigma_3$, and $\sigma_3\sigma_4$ respectively then 
 $a\mapsto \sigma_1\sigma_2$, $b\mapsto \sigma_2\sigma_3$, $c\mapsto \theta_2 \sigma_3\sigma_4\theta_2^{-1}$, $d\mapsto \theta_2\theta_3\sigma_4\sigma_5\theta_3^ {-1}\theta_2^{-1}$ and $x_0\mapsto \theta_2\theta_3 r\theta_1$ gives a representation of $\Gamma_0$.
 After tidying up these computations we obtain the moduli in \cref{teoModuli}.
\end{remark}

\begin{remark}
Given a family $\{w_i\}_{i\in I}$ of words in $F(a,b,c,d,x_0,\ldots,x_k)$,
the set of points $\mathbf{z} \in \C^6\times \SO(3,\C)^k\subseteq \C^{6+9k}$ such that $\rho_{\mathbf{z}}(w_i)=1$ for all $i\in I$ is an affine algebraic variety that we denote $Z(\{w_i\tq i\in I\})$.
For $k=0$ the variety $Z(w_0)$ can be described with only $6$ equations.
More generally, if we allow $X_1,\ldots,X_k$ to take values in $\O(3,\C)$ the variety $Z(w_0,\ldots, w_k)$ can be described using $6+9k$ equations.
This suggests that it may be possible to use a result such as B\'ezout's theorem to count points.
We could not finish this approach so we took a different one.
\end{remark}

\begin{proposition}\label{exactlyOneBad}
There is exactly one choice of
$$(\alpha_1,\beta_1,\alpha_2,\beta_2,\alpha_3,\beta_3,X_1,\ldots, X_k)\in {\C^6\times \SO(3,\C)^k}$$
with $\alpha_i^2+\beta_i^2=1$ for $i=1,2,3$ 
such that the matrices in \cref{teoModuli} satisfy
$$X_0=X_1=\ldots =X_k=(BAC)^3=1.$$
The unique solution $$\zzbad=(\alpha_1^{\bad},\beta_1^{\bad},\alpha_2^{\bad},\beta_2^{\bad},\alpha_3^{\bad},\beta_3^{\bad},1,\ldots, 1)$$
is real and its exact value is given by
\begin{center}
{\setlength{\tabcolsep}{16pt}
\def\arraystretch{2}
\begin{tabular}{lll}
 $\alpha_1^{\bad}= -\frac{1}{4}\sqrt{3\sqrt{5} + 9}$ & $\alpha_2^{\bad}= -\sqrt{-\frac{2}{15}\sqrt{5} + \frac{1}{3}}$ & $\alpha_3^{\bad}= -\sqrt{-\frac{1}{5}\sqrt{5} + \frac{1}{2}}$ \\
 $\beta_1^{\bad}= \frac{1}{4}\sqrt{-3\sqrt{5} + 7}$ & $\beta_2^{\bad}= \sqrt{\frac{2}{15}\sqrt{5} + \frac{2}{3}}$ & $\beta_3^{\bad} =  \sqrt{\frac{1}{5}\sqrt{5} + \frac{1}{2}}$.
\end{tabular}
}
\end{center}

\begin{proof}
Again this reduces to the case $k=0$.
In \cref{appendixSAGE} we give a proof using SAGE.
We indicate here how to prove this by hand.
We rewrite
$(BAC)^3=1$ as 
\begin{align}\label{eqBAC3}
(BAC)^2 - (BAC)^T &= 0 
\end{align}
and $X_0=1$ as
\begin{align}\label{eqX0}
S_4\rotMat(\alpha_1,\beta_1)S_1\rotMat(\alpha_2,\beta_2) - \rotMat(\alpha_3,\beta_3)^T S_3^T&=0.
\end{align}
We have 
$$BAC=\begin{pmatrix}
-\frac{\sqrt{3}}{3}\alpha_1\beta_1 - \frac{1}{2}(\alpha_1^2 - \beta_1^2) &       \alpha_1\beta_1 - \frac{\sqrt{3}}{6}(\alpha_1^2 - \beta_1^2) & -\frac{\sqrt{6}}{3} \\
 -\frac{1}{3}\alpha_1\beta_1 + \frac{\sqrt{3}}{2}(\alpha_1^2 - \beta_1^2) &  -\sqrt{3}\alpha_1\beta_1 - \frac{1}{6}(\alpha_1^2 - \beta_1^2) & -\frac{\sqrt{2}}{3} \\
 -\frac{4\sqrt{2}}{3}\alpha_1\beta_1 & -\frac{2\sqrt{2}}{3}(\alpha_1   ^2 - \beta_1^2) & \frac{1}{3}
\end{pmatrix}.$$
Then the $(3,3)$ entry of (\ref{eqBAC3}) gives
$\frac{8\sqrt{3}}{9}\alpha_1\beta_1 + \frac{4}{9}(\alpha_1^2 - \beta_1^2) - \frac{2}{9}=0$
and from $\alpha_1^2+\beta_1^2=1$ we obtain
\begin{align}\label{eqBAC333}
\frac{8\sqrt{3}}{9}\alpha_1\beta_1 + \frac{8}{9}\alpha_1^2 - \frac{2}{3} &=0.
\end{align}
To find the entries of (\ref{eqX0}) it is useful to recall that
$\cos(\frac{2\pi}{5})=\frac{1}{4}(\sqrt{5} -1)$, 
$\sin(\frac{2\pi}{5})=\frac{1}{4}\sqrt{2\sqrt{5} + 10}$, 
$\cos(\frac{\pi}{5})=\frac{1}{4}(\sqrt{5} +1)$, and 
$\sin(\frac{\pi}{5})=\frac{1}{4}\sqrt{-2\sqrt{5} + 10}$.
We obtain
$$\begin{pmatrix}
   -\frac{\sqrt{3}}{3}\alpha_2\beta_1- \frac{\sqrt{6}}{3}\beta_2 + \frac{\sqrt{5} + 1}{4}\beta_3  & -\frac{\sqrt{3}}{3}\beta_1\beta_2 + \frac{\sqrt{6}}{3}\alpha_2 - \alpha_3 & \frac{\sqrt{3}}{3}\alpha_1 + \frac{\sqrt{-2\sqrt{5} + 10}}{4}\beta_3 \\
-\alpha_1\alpha_2 - \frac{\sqrt{5} + 1}{4}\alpha_3 & -\alpha_1\beta_2 - \beta_3 & - \beta_1 -\frac{\sqrt{-2\sqrt{5} + 10}}{4}\alpha_3  \\
-\frac{\sqrt{6}}{3}\alpha_2\beta_1 + \frac{\sqrt{3}}{3}\beta_2 - \frac{\sqrt{-2\sqrt{5} + 10}}{4} &-\frac{\sqrt{6}}{3}\beta_1\beta_2 - \frac{\sqrt{3}}{3}\alpha_2 & \frac{\sqrt{6}}{3}\alpha_1 + \frac{\sqrt{5}+1}{4}
  \end{pmatrix}=0.
$$
The $(3,3)$ entry determines the value of $\alpha_1$.
From (\ref{eqBAC333}) we obtain the value of $\beta_1$.
The $(1,3)$ entry allows to obtain the value of $\beta_3$.
Now the $(2,2)$ entry gives the value of $\beta_2$.
The $(3,2)$ entry gives the value of $\alpha_2$ and finally the $(2,3)$ entry determines the value of $\alpha_3$.
Computing the remaining entries we see these values form a solution to $X_0=(BAC)^3=1$ and satisfy $\alpha_i^2+\beta_i^2=1$.
\end{proof}

\end{proposition}

\begin{remark}
 We say that $\zzbad$ is \textit{universal} in the following sense: if $\{w_i\}_{i\in I}\subseteq \ker(\phi)$ then $\zzbad\in Z(\{w_i\tq i\in I\})$. 
\end{remark}

The following result is proved in \cref{sectionQuaternions}.

\begin{theorem}\label{atLeastTwo}
Let $w_0,\ldots,w_k \in \ker(\phi)$. Let $N=\ker(\overline{\phi})$. If 
$N = \llangle w_0,\ldots, w_k\rrangle^{\Gamma_k} [N,N]$
then the variety $Z(w_0,\ldots,w_k)$ has at least two different points. 
\end{theorem}
Note that, by \cref{prop3by3}, the condition $N = \llangle w_0,\ldots, w_k\rrangle^{\Gamma_k} [N,N]$ is equivalent to the acyclicity of the corresponding $2$-complex. This is also the same as saying that $w_0,\ldots,w_k$ generate the $A_5$-module $N/[N,N]$ (i.e. the \textit{relation module} of $1\to N\to \Gamma_k\xrightarrow{\overline{\phi}} A_5\to 1$).

\section{Quaternions}
\label{sectionQuaternions}

To prove \cref{atLeastTwo} we study the real part of the moduli, working with quaternions instead of orthogonal matrices.
This is useful because representing a rotation as a quaternion allows to find the axis easily.

Recall that $S^3=\{ q\in \H \tq |q|=1\}$ acts on $S^2=\{ b\ii+c\jj+d\kk \tq b^2+c^2+d^2=1\}$ by conjugation.
Recall that any element of $S^3$ can be written as $\cos(\theta/2)+\sin(\theta/2) q$ with $\theta\in[0,2\pi]$ and $q=b\ii+c\jj+d\kk\in S^2$.
There is a homomorphism $p\colon S^3\to \SO(3,\R)$ with $\ker(p)=\{1,-1\}$ and which sends
$\cos(\frac{\theta}{2})+\sin(\frac{\theta}{2}) (b\ii+c\jj+d\kk)$
to the rotation matrix with angle $\theta$ and axis $(b,c,d)$.
Note that $\wt{\rotMat}(t)=\cos(\frac{t}{2})+\kk \sin(\frac{t}{2})$ is a lift of $\rotMat(\cos(t),\sin(t))$ by $p$.

Let $\psi\colon \H\to \R^3$ be given by $a+b\ii+c\jj+d\kk\mapsto (b,c,d)$.
Recall that if $q\in S^3$ and $v$ is a pure quaternion we have $\psi(q v q^{-1}) = p(q)\cdot \psi(v)$.
Let $\D^3\subset \R^3$ be the unit disk.
Let $\varphi\colon \D^3 \to \H$ be given by $(b,c,d)\mapsto \sqrt{1-b^2-c^2-d^2}+ b\ii+c\jj+d\kk$.
We denote the coordinates of $[0,2\pi]^3\times (\D^3)^k$ by $t_1,t_2,t_3,\ldots, t_{3(k+1)}$.

\begin{definition}\label{defRepQuaternions}
Let $\wt{A}$, $\wt{B}$, $\wt{S}_0$, $\wt{S}_1$, $\wt{S}_2$, $\wt{S}_3$, $\wt{S}_4$, 
be preimages by $p$ of the matrices $A$, $B$, $S_0$, $S_1$, $S_2$, $S_3$, $S_4$ which appear in the statement of \cref{teoModuli}.
We also define functions $\wt{C},\wt{D},\wt{X}_0\colon [0,2\pi]^3 \times (\D^3)^k \to \H$ by
\begin{align*}
\wt{C}(\tt) &= \wt{\rotMat}(t_1) \,\wt{S}_0\,  \wt{\rotMat}(t_1)^{-1},\\
\wt{D}(\tt) &= \wt{\rotMat}(t_1)\, \wt{S}_1\, \wt{\rotMat}(t_2)\, \wt{S}_2\, \wt{\rotMat}(t_2)^{-1}\, \wt{S}_1^{-1}\, \wt{\rotMat}(t_1)^{-1},\\
\wt{X}_0(\tt) &= \wt{\rotMat}(t_1)\, \wt{S}_1\, \wt{\rotMat}(t_2) \, \wt{S}_3 \, \wt{\rotMat}(t_3)\, \wt{S}_4.
\end{align*}
For $i=1,\ldots, k$ we define $\wt{X}_i( \tt ) = \varphi( t_{3i+1},t_{3i+2}, t_{3i+3})$.
Let $t_1^{\bad},t_2^{\bad},t_3^{\bad} \in [0,2\pi]^3$ be the unique numbers such that $\cos(t_i^{\bad})=\alpha_i^{\bad}$ and $\sin(t_i^{\bad})=\beta_i^{\bad}$.
Let $\ttbad = (t_1^{\bad},t_2^{\bad},t_3^{\bad},0,\ldots, 0)\in [0,2\pi]^3\times (\D^3)^k$.
Note that we can arrange the signs of these preimages so that $\left(\wt{B}\wt{A}\wt{C}\right)^3(\ttbad)=1$ and $\wt{X}_0(\ttbad)=1$.
\end{definition}

If $w\in F(a,b,c,d,x_0,\ldots, x_k)$ there is an induced map $\wt{W}\colon [0,2\pi]^3\times (\D^3)^k \to S^3$.
Note that any two words $w,w'$ which are equal in $\Gamma_k$ induce maps $\wt{W}, \wt{W}'$ which are equal or differ on a sign.
If $w_0,\ldots, w_k\in\ker(\phi)$ we can consider
$$\wt{\mathbf{W}}=(\wt{W_0},\ldots,\wt{W_k})\colon [0,2\pi]^3\times (\D^3)^k \to (S^3)^{k+1}$$
which can be composed with
$$\Psi = (\psi,\ldots, \psi) \colon \H^{k+1}\to \R^{3(k+1)}$$
to obtain a map 
$$\Psi\, \wt{\mathbf{W}} \colon [0,2\pi]^3\times (\D^3)^k \to (\D^3)^{k+1}.$$

The plan is to assume $\ttbad$ is the only zero in order to do a degree argument.
We will get a contradiction by computing the degree in two different ways.

\begin{lemma}\label{lemmaDegreeEvenGeneralized}
Let $I=[-1,1]$ and let $\D^3\subset \R^3$ be the unit disk.
Let $$\mathbf{F}=(f_0,\ldots,f_k) \colon I^3\times (\D^3)^k \to  (\D^3)^{k+1}$$ be a continuous map 
which is nonzero on the boundary of $I^3\times (\D^3)^k$ and satisfies the following parity condition:
\begin{itemize}
 \item For $t_1,t_2,t_3\in I$, $x_1,\ldots, x_k \in \D^3$ we have
\begin{align*}
(f_0,f_1,\ldots,f_k)((-1,t_2,t_3),x_1,\ldots,x_k)&= (-f_0,f_1,\ldots,f_k)((1,t_2,t_3),x_1,\ldots, x_k)\\
(f_0,f_1,\ldots,f_k)((t_1,-1,t_3),x_1,\ldots,x_k)&= (-f_0,f_1,\ldots,f_k)((t_1,1,t_3),x_1,\ldots, x_k)\\
(f_0,f_1,\ldots,f_k)((t_1,t_2,-1),x_1,\ldots,x_k)&= (-f_0,f_1,\ldots,f_k)((t_1,t_2,1),x_1,\ldots, x_k).
\end{align*}
\item For each $1\leq i\leq k$ and for every $(x_0,\ldots, x_k)\in I^3\times (\D^3)^k$ with $x_i\in \partial \D^3$ we have
$$(f_0,f_1,\ldots,f_k)(x_0 ,\ldots,x_{i-1},-x_i,x_{i+1},\ldots,x_k)= (f_0 ,\ldots,f_{i-1},-f_i,f_{i+1},\ldots,f_k)(x_0,\ldots, x_k).$$
\end{itemize}
Then the restriction
$$\mathbf{F}\colon {\partial(I^3\times (\D^3)^k )}\to (\D^3)^{k+1} -\{0\}$$
has even degree.

\begin{proof}
We fix cellular structures.
For $I$ we take the structure with two $0$-cells and one $1$-cell.
For $\D^3$ we take the cell structure with two $0$-cells, two $1$-cells, two $2$-cells and one $3$-cell (the antipodal map interchanges the $i$-cells in each pair for $0\leq i\leq 2$).
We take the product cellular structure for $I^3$, $I^3\times (\D^3)^k$ and $(\D^3)^{k+1}$.
Let $S=\partial(I^3\times (\D^3)^k)$.
Note that the $(3k+2)$-cells of $S$ can be divided into $3+k$ pairs of \textit{opposite} cells in a natural way.
Note that it is easy to define a cellular map $h_0\colon I^3\to \partial\, \D^3$ which satisfies
\begin{align*}
h_0(-1,t_2,t_3)&=-h_0(1,t_2,t_3)\\
h_0(t_1,-1,t_3)&=-h_0(t_1,1,t_3)\\
h_0(t_1,t_2,-1)&=-h_0(t_1,t_2,1).
\end{align*}
Let $h_i\colon \D^3\to \D^3$ be the identity  for $1\leq i \leq k$.
Now we can define a homotopy between $\mathbf{F}|_S$ and a map $\mathbf{G}\colon S\to \partial (\D^3)^{k+1}$
that satisfies the parity condition and coincides with 
$\mathbf{H}=(h_0,\ldots,h_k)$ on the $(3k+1)$-skeleton of $S$.
This is done skeleton by skeleton using that $\partial(\D^3)^{k+1}$ is $(3k+1)$-connected.
For each pair of opposite $(3k+2)$-cells we can extend the homotopy so that the parity condition is also satisfied by $\mathbf{G}$.
Clearly the degrees of $\mathbf{F}|_S$ and $\mathbf{G}$ are equal.
Now note that if $e,e'$ is a pair of opposite $(3k+2)$-cells then $\mathbf{G}_*(e),\mathbf{H}_*(e)\in C_{3k+2}(\partial(\D^3)^{k+1})$ differ on an element of $H_{3k+2}(\partial (\D^3)^{k+1})$.
Moreover, by the parity condition, $\mathbf{G}_*(e')$ and $\mathbf{H}_*(e')$ differ on the same element.
Thus the degree of $\mathbf{H}|_{S}$ and the degree of $\mathbf{G}$ are equal modulo $2$.
To conclude, note that $\deg(\mathbf{H}|_{S})=0$ since $\mathbf{H}\colon I^3\times (\D^3)^k\to \partial (\D^3)^{k+1} $ is an extension of $\mathbf{H}|_{S}$ to a contractible space.
\end{proof}
\end{lemma}

\begin{corollary}\label{degreeWIsEven}
Let $w_0,\ldots,w_k\in F(a,b,c,d,x_0,\ldots,x_k)$ be words and assume the total exponent of $x_i$ in $w_j$ is $\delta_{i,j}$.
If $\Psi\wt{\mathbf{W}}$ is nonzero on the boundary of $[0,2\pi]^3\times (\D^3)^{k}$, then the degree of the restriction
$\Psi \wt{\mathbf{W}}\colon \partial \left( [0,2\pi]^3\times (\D^3)^{k}\right)\to (\D^3)^{k+1}-\{0\}$
is even.
\begin{proof}
Since the total exponent of $x_i$ in $w_j$ is $\delta_{i,j}$, by looking at \cref{defRepQuaternions} we see the parity condition of \cref{lemmaDegreeEvenGeneralized} is satisfied.
 \end{proof}
\end{corollary}

Recall that the degree can be computed in the following way
\begin{lemma}\label{lemmaDegreeAndDifferential}
 Let $f\colon \R^n\to \R^n$ be smooth and assume $f(0)=0$. If $\det(Df_0)\neq 0$ then $0$ is an isolated zero and the degree of $f$ around $0$ is given by $\deg(f,0)=\sg(\det(Df_0))$.
\end{lemma}

We need some basic differentiation properties for quaternion valued analytic functions analogous to the usual ones (see \cref{appendixQuaternions}).
Note that $\wt{\rotMat}(t)= \cos(\frac{t}{2})+\kk \sin(\frac{t}{2}) = 1 + \frac{t}{2} \kk + O(t^2)$.

\begin{lemma}\label{lemmaMxInvertible}
  Let $\wt{\mathbf{X}} = (\wt{X}_0,\ldots, \wt{X}_k)$.
  Then $D\left(\Psi \wt{\mathbf{X}}\right)_{\ttbad}$ is invertible.
\begin{proof}
Again this reduces to the case $k=0$ by noting that 
$$D\left(\Psi \wt{\mathbf{X}}\right)_{\ttbad} = 
\begin{pmatrix}
M & 0 \\
0 & I
\end{pmatrix}                                                 
$$
where $M$ is the $3\times 3$ matrix we obtain in the $k=0$ case.
We now prove $M$ is invertible.
Recall that $\wt{X}_0(\ttbad)=1$. Then
  \begin{align*}
\wt{X}_0(\ttbad+\tt) =&\, \wt{\rotMat}(t^{\bad}_1) \, \wt{\rotMat}(t_1)\, \wt{S}_1\,  \wt{\rotMat}(t_2^{\bad})\, \wt{\rotMat}(t_2)\, \wt{S}_3\, \wt{\rotMat}(t_3^{\bad})\, \wt{\rotMat}(t_3)\, \wt{S}_4 \\
 =&\, \wt{\rotMat}(t^{\bad}_1)  \left(1+\frac{t_1}{2}\kk\right)  \wt{S}_1\, \wt{\rotMat}(t_2^{\bad}) \left(1+\frac{t_2}{2}\kk\right)  \wt{S}_3\, \wt{\rotMat}(t_3^{\bad}) \left(1+\frac{t_3}{2}\kk\right)  \wt{S}_4\, +\,  O(\tt^2) \\
 =&\, 1\,  + \,   \frac{1}{2}\, \wt{\rotMat}(t^{\bad}_1)\,  \kk \, \wt{\rotMat}(t^{\bad}_1)^{-1} \, t_1 \,  + \, \frac{1}{2}\, \left(\wt{\rotMat}(t^{\bad}_1)\, \wt{S}_1\, \wt{\rotMat}(t_2^{\bad})\right)\,  \kk\,  \left(\wt{\rotMat}(t^{\bad}_1)\, \wt{S}_1\, \wt{\rotMat}(t_2^{\bad})\right)^{-1}  t_2\, \\
  &\, + \, \frac{1}{2}\,  \wt{S}_4^{-1} \, \kk\,  \wt{S}_4 \, t_3 \, +\,  O(\tt^2)
  \end{align*}
  Now recalling that $q\, \kk\, q^{-1} = (\ii,\jj,\kk)\cdot  p(q)\cdot (0,0,1)$ for any $q\in S^3$ 
  we see that the columns of $M$ are given by
  \begin{align*}
  \frac{1}{2}\, \rotMat(\alpha_1^{\bad},\beta_1^{\bad}) \cdot (0,0,1) &=  \left(0,0,\frac{1}{2}\right)\\
  \frac{1}{2}\, \rotMat(\alpha_1^{\bad},\beta_1^{\bad})\, S_1\, \rotMat(\alpha_2^{\bad},\beta_2^{\bad}) \cdot (0,0,1) &=\left(-\frac{1}{2}\beta_1^{\bad},-\frac{1}{2} \alpha_1^{\bad},0\right)\\
  \frac{1}{2}\, S_4^{-1} \cdot (0,0,1)&= \left(0,-\frac{1}{6}\sqrt{6},\frac{1}{6}\sqrt{3}\right).
  \end{align*}
  Thus
  $$M=
  \begin{pmatrix}
      0 & -\frac{1}{2}\beta_1^{\bad} & 0 \\
      0 & -\frac{1}{2}\alpha_1^{\bad} & -\frac{1}{6}\sqrt{6} \\
      \frac{1}{2} & 0 & \frac{1}{6}\sqrt{3}
  \end{pmatrix}$$
and therefore $\det(M)= \frac{1}{24}\sqrt{6}\beta_1^{\bad}\neq 0$.
\end{proof}
\end{lemma}

\begin{lemma}\label{lemmaDerivativeIsPure}
Let $w\in \ker(\phi)$. Then $\frac{\partial \wt{W}}{\partial t_i}(\ttbad)$ is a pure quaternion for $i=1,\ldots, 3(k+1)$.

\begin{proof}
 Since $w$ belongs to $\ker(\phi)$, in $\Gamma_k$ it equals a product of conjugates of the $x_i$, $(bac)^3$ and their inverses.
 Recall that $S^2$ is invariant by the action of $S^3$.
 By \cref{propProductRule}, it is enough to prove that
 $\frac{\partial \wt{X}_j}{\partial t_i}(\ttbad)$ and $\frac{\partial (\wt{B}\wt{A}\wt{C})^3}{\partial t_i}(\ttbad)$
 are pure quaternions.
 
 For $i=0$ the first claim follows from the computation in the proof of \cref{lemmaMxInvertible} and
 is easy to verify for $i>0$.
 The second claim follows similarly by noting that $\left(\wt{B}\wt{A}\wt{C}\right)^3(\ttbad)=1$ and writing
\begin{align*}
    \left(\wt{B}\wt{A}\wt{C}\right)^3(\ttbad+\tt) &= \left(\wt{B}\wt{A}\left(1+\frac{t_1}{2}\kk\right)\wt{S}_0\left(1-\frac{t_1}{2}\kk\right)\right)^3+O(\tt^2).
 \end{align*}
\end{proof}

\end{lemma}

\begin{lemma}\label{lemmaDifferentialW}
  Let $N=\ker(\overline{\phi})$ and let $w_0,\ldots, w_k \in \ker(\phi)$. If
  $N = \llangle w_0,\ldots,w_k\rrangle^{\Gamma_k}[N,N]$ 
  then $D\left(\Psi \wt{\mathbf{W}}\right)_{\ttbad}$ is invertible.
\begin{proof}
We may assume without loss of generality that $\wt{W}_{j}(\ttbad)=1$ for all $j$.
For each $j$ there are numbers $a_j, \ell_j\in \N_0$,
words $v_{j,1},\ldots, v_{j,a_j}, u_{j,1},\ldots, v_{j,a_j}\in \ker(\phi)$,
words $p_{j,1},\ldots p_{j,\ell_j}\in F(a,b,c,d,x_0,\ldots,x_k)$,
indices $\alpha_{j,1},\ldots, \alpha_{j,{\ell_j}}\in \{0,\ldots, k\}$
and signs $\varepsilon_{j,1},\ldots,\varepsilon_{j,\ell_j}\in\{1,-1\}$ such that in $\Gamma_k$ we have
$$x_j  = \prod_{s=1}^{\ell_j} p_{j,s} w_{\alpha_{j,s}}^{\varepsilon_{j,s}} p_{j,s}^{-1} \prod_{i=1}^{a_j}[u_{j,i},v_{j,i}].$$
Then we have
  $$\wt{X}_j(\ttbad+\tt)=\left(\prod_{s=1}^{\ell_j} \wt{P}_{j,s} \wt{W}_{\alpha_{j,s}}^{\varepsilon_{j,s}} \wt{P}_{j,s}^{-1} \prod_{i=1}^{a_j}[\wt{U}_{j,i},\wt{V}_{j,i}] \right)(\ttbad+\tt)$$
and using \cref{propProductRule} we obtain
  \begin{align*}
  \frac{\partial \wt{X}_j}{\partial t_i}(\ttbad)  &= \sum_{s=1}^{\ell_j}\wt{P}_{j,s}(\ttbad) \frac{\partial \wt{W}_{\alpha_{j,s}}^{\varepsilon_{j,s}}}{\partial t_i}(\ttbad) \wt{P}_{j,s}^{-1}(\ttbad)\\
  &= \sum_{s=1}^{\ell_j} \varepsilon_{j,s}\wt{P}_{j,s}(\ttbad) \frac{\partial \wt{W}_{\alpha_{j,s}}}{\partial t_i}(\ttbad) \wt{P}_{j,s}^{-1}(\ttbad)
  \end{align*}
By \cref{lemmaDerivativeIsPure},
$D\left(\Psi \wt{\mathbf{W}}\right)_{\ttbad}$
is invertible if and only if
$$\left\{ \left( \frac{ \partial \wt{W}_0 }{\partial t_i}(\ttbad) ,\ldots,\frac{ \partial \wt{W}_{k} }{\partial t_i}(\ttbad) \right) 
\tq  1\leq i \leq 3(k+1)   \right\}$$
is linearly independent over $\R$.
If $\lambda_i\in\R$ satisfy 
$$\sum_{i=1}^{3(k+1)} \lambda_i \frac{\partial \wt{\mathbf{W}}}{\partial t_i}(\ttbad)  = 0$$
it follows that
$$\sum_{i=1}^{3(k+1)} \lambda_i \frac{\partial \wt{\mathbf{X}}}{\partial t_i}(\ttbad)  = 0.$$
By \cref{lemmaMxInvertible}, $D\left(\Psi \wt{\mathbf{X}}\right)_{\ttbad}$
is invertible and, again by \cref{lemmaDerivativeIsPure}, the set 
$$\left\{ \left( \frac{ \partial \wt{X}_0 }{\partial t_i}(\ttbad) ,\ldots,\frac{ \partial \wt{X}_{k} }{\partial t_i}(\ttbad) \right) 
\tq  1\leq i \leq 3(k+1)   \right\}$$
is linearly independent over $\R$.
Thus $\lambda_1 = \ldots = \lambda_{3(k+1)}=0$ and we are done.
\end{proof}
\end{lemma}

\begin{proof}[Proof of \cref{atLeastTwo}]
  We can assume that the total exponent of $x_i$ in $w_j$ is $\delta_{i,j}$. 
  To prove this, consider the abelianization and note that it is possible to achieve this by using the following operations:
  \begin{itemize}
    \item replacing $w_i$ by $w_iw_j$ (if $i\neq j$), 
    \item replacing $w_i$ by $w_i^{-1}$, and
    \item interchanging $w_i$ and $w_j$.
  \end{itemize}
  By \cref{lemmaDifferentialW} and \cref{lemmaDegreeAndDifferential}, the degree of  $\Psi\wt{\mathbf{W}}$ near $\ttbad$ is $\pm 1$.
  If $\Psi\wt{\mathbf{W}}$ has a zero on $\partial([0,2\pi]^3\times (\D^3)^k)$ we are done. Otherwise, by \cref{degreeWIsEven}, the degree of $\Psi\wt{\mathbf{W}}$ restricted to the boundary of $[0,2\pi]^3\times (\D^3)^k$ is even.
  It follows that there must be a point $\tt \neq \ttbad$ such that $\Psi\wt{\mathbf{W}}(\tt)=0$.
  This gives a second point in $Z(w_0,\ldots, w_k)$.
\end{proof}

\section{Group actions of $A_5$ on contractible $2$-complexes}
\label{sectionConsequences}
We can now prove the following.

\begin{theorem}\label{theoremWi}
\statementTheoremWi
\begin{proof}
This follows from \cref{atLeastTwo}, \cref{exactlyOneBad} and \cref{goodRep}.
\end{proof}
\end{theorem}

Now from \cref{theoremWi} and \cref{TheoremEquivalenceA5CDAndWi} we deduce.
\begin{theorem}\label{theoremCDA5}
\statementTheoremCDAFive
\end{theorem}

\begin{corollary}\label{CDForSimpleGroupsContainingA5}
 \statementCDForSimpleGroupsContainingAFive
 \begin{proof}
  Let $G$ be one of these groups and let $X$ be a finite acyclic fixed point free $2$-dimensional $G$-complex.  
  By \cite[Proposition 3.3]{OS}, $A_5$ is a subgroup of $G$ and by \cref{lemmaNonsolvableSubgroup} the action of $A_5$ on $X$ is fixed point free.
  By \cref{theoremCDA5} $X$ cannot be contractible.
 \end{proof}

\end{corollary}

Looking more carefully at the proof of \cref{theoremCDA5} we obtain the following.

\begin{theorem}\label{A5ComplexHasRepInSO3}
 Let $X$ be a fixed point free $2$-dimensional finite, acyclic $A_5$-complex.
 Then there is a nontrivial representation $\pi_1(X) \to \SO(3,\R)$.
  \begin{proof}
  By \cref{Refinement} we see that $\pi$ surjects onto the fundamental group of an acyclic $2$-dimensional $A_5$ complex $X'$ which is obtained from $\XOS(A_5)$ by attaching $k\geq 0$ free orbits of $1$-cells and $k+1$ free orbits of $2$-cells.
Now note that the representation constructed to prove \cref{theoremWi} restricted to $\pi_1(X')$ gives a nontrivial morphism into $\SO(3,\R)$.
 \end{proof}
\end{theorem}

Since $A_5$ is the only finite perfect subgroup of $\SO(3,\R)$ we deduce the following.
\begin{corollary}\label{fundamentalGroupsA5}
\statementCharacterizationFundamentalGroupsAcyclicAFive
\end{corollary}

Recall that $N=\ker\left(\overline{\phi}\right)$ is a free group of rank $60(k+1)$.
We can restate \cref{theoremWi} in the following way which highlights the connection with the relation gap problem (see \cite{HarlanderRelationGap1,HarlanderRelationGap2}).
\begin{corollary}
The extension
$$1\to N \to \Gamma_k \xrightarrow{\overline{\phi}} A_5\to 1$$ 
has a relation gap. 
That is, the $A_5$-module $N/[N,N]$ is free of rank $k+1$.
However $N$ cannot be generated by $k+1$ elements as a $\Gamma_k$-group.
\end{corollary}
Note that since $\Gamma_k$ is not free this is not an example of a presentation with a relation gap.

\appendix

\section{Equations over groups}
\label{appendixEquationsOverGroups}
Let $G$ be a group. An \textit{equation} over $G$ in the variables $x_1,\ldots,x_n$ is an element $w \in G*\freeGroup(x_1,\ldots,x_n)$.
We say that a system of equations
\begin{align*}
 w_1(x_1,\ldots,x_n) &= 1 \\
 w_2(x_1,\ldots,x_n) &= 1 \\
 \cdots\hspace{19pt} & \\
 w_m(x_1,\ldots,x_n) &= 1
\end{align*}
\textit{has a solution in an overgroup} of $G$ if the map $G\to G*\freeGroup(x_1,\ldots,x_m)/\llangle w_1,\ldots,w_m\rrangle$ is injective.
Such a system of equations determines an $(m\times n)$-matrix $M$ where $M_{i,j}$ is given by the total exponent of the letter $x_j$ in the word $w_i$.
A system is said to be \textit{independent} if the rank of $M$ is $m$.

One of the most important open problems in the theory of equations over groups is the Kervaire--Laudenbach--Howie conjecture \cite[Conjecture]{HowieEquationsOverGroups}.

\begin{conjecture}[Kervaire--Laudenbach--Howie]\label{KLHConjecture}
An independent system of equations over $G$ has a solution in an overgroup of $G$.
\end{conjecture}

The Gerstenhaber--Rothaus theorem \cite[Theorem 3]{GerstenhaberRothaus} says that finitely generated subgroups of compact connected Lie groups satisfy \cref{KLHConjecture}.

\begin{proposition}\label{propSystemEquations}
 Let $X$ be a finite acyclic $2$-complex and let $A\subset X$ be an acyclic subcomplex.
 Then we can write
 $$\pi_1(X)=\pi_1(A)*F(x_1,\ldots,x_n)/\llangle w_1,\ldots, w_n\rrangle$$
 and the $(n\times n)$-matrix $M$ such that $M_{i,j}$ is the total exponent of $x_j$ in $w_i$ is invertible.
\begin{proof}
Take a maximal tree $T$ for $A$ and consider a maximal tree $\overline{T}$ of $X$ containing $T$.
Then $A/T\simeq A$ is an acyclic subcomplex of the acyclic $2$-complex $X/\overline{T}\simeq X$.
As usual, from $A/T$ we can read a presentation for $\pi_1(A)$ which is balanced since $A/T$ is acyclic.
Now we consider a variable $x_i$ for each $1$-cell of $X/\overline{T}$ which is not in $A/T$ and we read words from the attaching maps for the $2$-cells of $X/\overline{T}$ which are not part of $A/T$.
In this way we obtain equations in these variables with coefficients in $A$ which give the desired description of $\pi_1(X)$.
Since $X/\overline{T}$ is acyclic, there is an equal number of variables and equations and the matrix $M$ is invertible.
\end{proof}
\end{proposition}

Now from the Gerstenhaber--Rothaus theorem we deduce.

\begin{proposition}\label{casoParticularConjeturaSubcomplejo}
Let $X$ be a finite acyclic $2$-complex.
If $A\subset X$ is an acyclic subcomplex and there is a nontrivial representation $\rho\colon \pi_1(A)\to \SO(n,\R)$ then $\pi_1(X)$ is nontrivial.
\begin{proof}
 By \cref{propSystemEquations} can write $\pi_1(X)=\pi_1(A) * F(x_1,\ldots, x_n) / \llangle w_1,\ldots, w_n\rrangle$ and the system is independent.
 Let $G_0=\rho(\pi_1(A))$.
 There is an induced map  $$\rho \colon \pi_1(A) * F(x_1,\ldots, x_n) \to G_0 * F(x_1,\ldots, x_n)$$ which induces an epimorphism
 $\pi_1(X)\to G_0*F(x_1,\ldots,x_n)/\llangle \rho(w_1),\ldots, \rho(w_n)\rrangle$.
 Finally from \cite[Theorem 3]{GerstenhaberRothaus} it follows that this group is nontrivial. 
 \end{proof}
\end{proposition}

\section{Quaternion valued analytic functions}
\label{appendixQuaternions}

A \textit{quaternion valued analytic function} is a function $f\colon U \to \H$ where $U\subset \R^n$ is open, such that its components are analytic, that is a function that can be written as $$f=f_1+f_\ii \ii + f_\jj \jj + f_\kk \kk$$ with
$f_1,f_\ii,f_\jj,f_\kk\colon U\to \R$ are analytic.
For $i=1,\ldots, n$ we can define the partial derivative
$$\frac{\partial f}{\partial t_i}= \frac{\partial f_1}{\partial t_i} +\frac{\partial f_\ii}{\partial t_i} \ii+\frac{\partial f_\jj}{\partial t_i} \jj+\frac{\partial f_\kk}{\partial t_i} \kk.$$
We define
$$Df_{\tt} = \left(\frac{\partial f}{\partial t_1}(\tt),\ldots, \frac{\partial f}{\partial t_n}(\tt)\right).$$
If each coordinate of $\mathbf{F}=(f_1,\ldots,f_m)\colon U\to \H^m$ is analytic then we use the notation
$$\frac{\partial \mathbf{F}}{\partial t_i}=\left(\frac{\partial f_1}{\partial t_i},\ldots, \frac{\partial f_m}{\partial t_i}\right).$$
The usual properties extend to this context. We need the following

\begin{proposition}\label{propProductRule}
Let $f,g\colon U \to \H$ be analytic. Then

 (i) We have the product rule $$\displaystyle\frac{\partial f\cdot g }{\partial t_i}( \tt )  = \frac{\partial f}{\partial t_i}( \tt) g(\tt)+ f(\tt)\frac{\partial g}{\partial t_i}(\tt).$$
 
 (ii) Suppose $f$ is nowhere zero and $g(\tt_0)\in \R$ then 
 $$\frac{\partial \, f\cdot g \cdot\displaystyle\frac{1}{f} }{\partial t_i}( \tt_0) = f(\tt_0) \frac{\partial g}{\partial t_i}(\tt_0) f(\tt_0)^{-1}.$$
 
 (iii) Suppose $f(\tt_0)=\pm 1$ then
 $$\displaystyle\frac{\partial{\frac{1}{f}}}{\partial t_i}(\tt_0) = - \frac{\partial f}{\partial t_i}(\tt_0).$$
 
 (iv) Suppose that $f,g$ are nowhere zero and $f(\tt_0),g(\tt_0)\in \{1,-1\}$.
 Then the commutator $[f,g]=f\cdot g \cdot \frac{1}{f} \cdot \frac{1}{g}$ satisfies
 $[f,g](\tt_0)=1$ and
 $$\frac{\partial [f,g]}{\partial t_i}(\tt_0)=0.$$
 \begin{proof}
  (i) is a straightforward computation, (ii) and (iii) follow from (i). Finally, (iv) follows from the previous properties.
 \end{proof}

\end{proposition}
As usual we have the Taylor series
$$f(\tt_0 + \tt ) = f(\tt_0)+ \sum_{i=1}^n \frac{\partial f}{\partial t_i}(\tt_0) t_i + O(\tt^2).$$
From the product rule we see that we can multiply the Taylor series of two functions to obtain the Taylor series of the product.


\section{Alternative proof of \cref{lemaXandBAC3}}
\label{appendixGAP}
An alternative way to finish the proof of \cref{lemaXandBAC3} goes by noting that $x=bc$ and $y=ca$ satisfy $x^2=y^5=(xy)^3=1$.
Thus it would suffice to to show the group
$$\langle a,b,c \mid a^2,b^3, c^2, (ab)^3, (bc)^2, (ca)^5, (bac)^3 \rangle$$
is generated by $x,y$,
for it is well known that $\langle x,y \mid x^2,y^5, (xy)^3\rangle$ is a presentation of $A_5$.
To do this, it is enough to show that $a\in \langle x,y\rangle$.
The following computation proves this claim.
\begin{align*}
xy^2 x y^{-2} x y &=  (bc) (ca) (ca) (bc) (a^{-1}c^{-1})(a^{-1}c^{-1}) (bc) (ca) \\
\text{\color{gray}(using $a^2=c^2=1$) } &= ba ca bc ac ac ba \\
\text{\color{gray}(replacing $acba$ by $b^{-1}cab^{-1}c$) }&= ba ca bc ac  b^{-1}cab^{-1}c\\
\text{\color{gray}(replacing $cb^{-1}c$ by $b$) }&= ba ca bc abab^{-1}c\\
\text{\color{gray}(replacing $aba$ by $b^{2}ab^{2}$) }&= ba ca bc b^{2}abc\\
\text{\color{gray}(replacing $bcb$ by $c$) }&= ba ca cbabc\\
\text{\color{gray}(replacing $acba$ by $b^{-1}cab^{-1}c$) }&= ba cb^{-1}cab^{-1}c bc\\
\text{\color{gray}(replacing $cb^{-1}c$ by $b$) }&= ba bab^{-1}c bc\\
\text{\color{gray}(replacing $cbc$ by $b^{-1}$) }&= ba bab^{-2}\\
\text{\color{gray}(using $b^{3}=1$) }&= ba bab\\
\text{\color{gray}(using $(ab)^3=1$) }&= a.
\end{align*}


\section{Alternative proofs using SAGE}
\label{appendixSAGE}
The following SAGE code gives alternative proofs of \cref{teoModuli} and \cref{exactlyOneBad} which are easier to verify.
Note that SAGE computes exactly over the algebraic numbers so there is no numerical error.
The function \verb|check_rep|,  shows $A,B,C,D,X_0$ satisfy the defining relations for $\Gamma_0$ in $$M_3(\C[\alpha_1,\beta_1,\alpha_2,\beta_2,\alpha_3,\beta_3]/\langle \alpha_1^2+\beta_1^2-1, \alpha_2^2+\beta_2^2-1, \alpha_3^2+\beta_3^2-1\rangle).$$
The function \verb|find_universal_representations| gives the exact value of the unique solution $\zzbad$ of $X_0=(BAC)^3=1$ by solving the corresponding system of polynomial equations over the algebraic closure of $\Q$.
Note that in this code we use $x_i$, $y_i$ instead of $\alpha_i$, $\beta_i$.

\begin{Verbatim}
def R(x,y):
	return matrix([
			(x,y,0,),
			(-y,x,0),
			(0,0,1),
		]);

A = matrix([
	(-1,0,0),
	(0,1/3,-2/3*sqrt(2)),
	(0,-2/3*sqrt(2),-1/3),
]);

B = matrix([
	(-1/2,-sqrt(3)/2,0),
	(sqrt(3)/2, -1/2,0),
	(0,0,1),
]);

S0 = matrix([
	(-1,0,0),
	(0,1,0),
	(0,0,-1)
])

S1 = matrix([
	(-1,0,0),
	(0,0,-1),
	(0,-1,0),
]);

S2 = matrix([
	(-cos(2*pi/5),0,-sin(2*pi/5)),
	(0,-1,0),
	(-sin(2*pi/5),0,cos(2*pi/5)),
]);

S3 = matrix([
	(0,cos(pi/5),sin(pi/5)),
	(1,0,0),
	(0,sin(pi/5),-cos(pi/5)),
]);

S4 = matrix([
	(0, -sqrt(3)/3, -sqrt(6)/3),
	(1,0,0),
	(0,-sqrt(6)/3,sqrt(3)/3),
]);

def rep(x1,y1,x2,y2,x3,y3):
	C = R(x1,y1) * S0 * R(x1,y1).T;
	D = R(x1,y1) * S1 * R(x2,y2) * S2 * R(x2,y2).T * S1.T * R(x1,y1).T;		
	X0 = R(x1,y1) * S1 * R(x2,y2) * S3 * R(x3,y3) * S4;	
	return (A,B,C,D,X0);

I = matrix.identity(3);

def check_rep():
	R.<x1,y1,x2,y2,x3,y3> = QQbar[];
	A,B,C,D,X0 = rep(x1,y1,x2,y2,x3,y3);
	J = R.ideal([x1^2+y1^2-1, x2^2+y2^2-1, x3^2+y3^2-1]);
	S = R.quotient(J);
	f = S.cover(); # f: R -> S
	M3R = MatrixSpace(R,3,3);
	M3S = MatrixSpace(S,3,3);
	M3f = M3R.hom(f,M3S); # M3f: M3R -> M3S
	A,B,C,D,X0 = M3f(A), M3f(B), M3f(C), M3f(D), M3f(X0);
	for M in [A,B,C,D,X0]:
		assert(M*transpose(M)==I);
		assert(M.det()==1);
	relations = [ A**2, (A*B)**3, B**3, (B*C)**2, C**2, 
				 (C*D)**5, D**2, X0*A*transpose(X0)*transpose(D) ];
	for r in relations:
		assert(r==I);
	print("The construction defines a representation of Gamma0.");

def find_universal_representations():
	R.<x1,y1,x2,y2,x3,y3> = QQbar[]
	A,B,C,D,X0 = rep(x1,y1,x2,y2,x3,y3);
	equations = [
		x1^2+y1^2-1,
		x2^2+y2^2-1,
		x3^2+y3^2-1 
	] + [
		M[i][j] for i in range(3)
				for j in range(3)
				for M in [ X0-I, (B*A*C)**3-I ] 
	];
	J = R.ideal(equations);
	dim_Z = J.dimension()
	assert(dim_Z==0)
	print("The variety of universal representations has dimension "
		      + str(dim_Z));
	Z = J.variety();
	assert(len(Z)==1);
	print("The number of universal representations is "+str(len(Z)));
	z_u = Z[0];
	print("The unique universal representation is given by:"); 
	for v in [x1,y1,x2,y2,x3,y3]:
		print(v, z_u[v].radical_expression() ,z_u[v], z_u[v].minpoly());
	return z_u;

\end{Verbatim}

\renewcommand{\thesection}{\arabic{section}}

\bibliographystyle{alpha}
\bibliography{references}

\begin{thebibliography}{GAP19}

\bibitem[Ade03]{Adem}
Alejandro Adem.
\newblock Finite group actions on acyclic 2-complexes.
\newblock {\em Ast\'erisque}, (290):Exp. No. 894, vii, 1--17, 2003.
\newblock S\'eminaire Bourbaki. Vol. 2001/2002.

\bibitem[AS93]{AschbacherSegev}
Michael Aschbacher and Yoav Segev.
\newblock A fixed point theorem for groups acting on finite {$2$}-dimensional
  acyclic simplicial complexes.
\newblock {\em Proc. London Math. Soc. (3)}, 67(2):329--354, 1993.

\bibitem[Bro84]{BrownPresentations}
Kenneth~S. Brown.
\newblock Presentations for groups acting on simply-connected complexes.
\newblock {\em J. Pure Appl. Algebra}, 32(1):1--10, 1984.

\bibitem[Bro06]{TopologyGroupoids}
Ronald Brown.
\newblock {\em Topology and groupoids}.
\newblock BookSurge, LLC, Charleston, SC, 2006.

\bibitem[CD92]{CD}
Carles Casacuberta and Warren Dicks.
\newblock On finite groups acting on acyclic complexes of dimension two.
\newblock {\em Publicacions Matem{\`a}tiques}, pages 463--466, 1992.

\bibitem[Cor92]{CorsonComplexesOfGroups}
Jon~M. Corson.
\newblock Complexes of groups.
\newblock {\em Proc. London Math. Soc. (3)}, 65(1):199--224, 1992.

\bibitem[Cor01]{Corson}
Jon~M. Corson.
\newblock On finite groups acting on contractible complexes of dimension two.
\newblock {\em Geom. Dedicata}, 87(1-3):161--166, 2001.

\bibitem[FR59]{FloydRichardson}
Edwin~E. Floyd and Roger~W. Richardson.
\newblock An action of a finite group on an {$n$}-cell without stationary
  points.
\newblock {\em Bull. Amer. Math. Soc.}, 65:73--76, 1959.

\bibitem[GAP19]{GAP}
{The GAP~Group}.
\newblock {\em {GAP -- Groups, Algorithms, and Programming, Version 4.10.1}},
  2019.

\bibitem[GR62]{GerstenhaberRothaus}
Murray Gerstenhaber and Oscar~S. Rothaus.
\newblock The solution of sets of equations in groups.
\newblock {\em Proc. Nat. Acad. Sci. U.S.A.}, 48:1531--1533, 1962.

\bibitem[Har15]{HarlanderRelationGap2}
Jens Harlander.
\newblock On the relation gap and relation lifting problem.
\newblock In {\em Groups {S}t {A}ndrews 2013}, volume 422 of {\em London Math.
  Soc. Lecture Note Ser.}, pages 278--285. Cambridge Univ. Press, Cambridge,
  2015.

\bibitem[Har18]{HarlanderRelationGap1}
Jens Harlander.
\newblock The relation gap problem.
\newblock In {\em Advances in two-dimensional homotopy and combinatorial group
  theory}, volume 446 of {\em London Math. Soc. Lecture Note Ser.}, pages
  128--148. Cambridge Univ. Press, Cambridge, 2018.

\bibitem[How81]{HowieEquationsOverGroups}
James Howie.
\newblock On pairs of {$2$}-complexes and systems of equations over groups.
\newblock {\em J. Reine Angew. Math.}, 324:165--174, 1981.

\bibitem[How02]{HowieScottWiegold}
James Howie.
\newblock A proof of the {S}cott-{W}iegold conjecture on free products of
  cyclic groups.
\newblock {\em J. Pure Appl. Algebra}, 173(2):167--176, 2002.

\bibitem[KLV01]{KLV}
Sava Krsti\'c, Martin Lustig, and Karen Vogtmann.
\newblock An equivariant {W}hitehead algorithm and conjugacy for roots of
  {D}ehn twist automorphisms.
\newblock {\em Proc. Edinb. Math. Soc. (2)}, 44(1):117--141, 2001.

\bibitem[KS79]{KirbyScharlemann}
Robion~C. Kirby and Martin~G. Scharlemann.
\newblock Eight faces of the {P}oincar\'{e} homology {$3$}-sphere.
\newblock {\em Geometric topology ({P}roc. {G}eorgia {T}opology {C}onf.,
  {A}thens, {G}a., 1977)}, pages 113--146, 1979.

\bibitem[Oli75]{OliverDisks}
Robert Oliver.
\newblock Fixed-point sets of group actions on finite acyclic complexes.
\newblock {\em Comment. Math. Helv.}, 50:155--177, 1975.

\bibitem[OS02]{OS}
Bob Oliver and Yoav Segev.
\newblock Fixed point free actions on {$\bf Z$}-acyclic 2-complexes.
\newblock {\em Acta Math.}, 189(2):203--285, 2002.

\bibitem[Sag19]{SAGE}
{The Sage Developers}.
\newblock {\em {S}ageMath, the {S}age {M}athematics {S}oftware {S}ystem
  ({V}ersion 8.9)}, 2019.
\newblock {\url{https://www.sagemath.org}}.

\bibitem[SC19]{SadofschiCostaThesis}
Iván Sadofschi~Costa.
\newblock {\em Fixed points of maps and actions on 2-complexes}.
\newblock PhD thesis, Universidad de Buenos Aires, 2019.
\newblock
  \url{https://bibliotecadigital.exactas.uba.ar/collection/tesis/document/tesis_n6591_SadofschiCosta}.

\bibitem[Seg93]{Segev}
Yoav Segev.
\newblock Group actions on finite acyclic simplicial complexes.
\newblock {\em Israel J. Math.}, 82(1-3):381--394, 1993.

\bibitem[Seg94]{SegevCollapsible}
Yoav Segev.
\newblock Some remarks on finite {$1$}-acyclic and collapsible complexes.
\newblock {\em J. Combin. Theory Ser. A}, 65(1):137--150, 1994.

\bibitem[Ser80]{SerreTrees}
Jean-Pierre Serre.
\newblock {\em Trees}.
\newblock Springer-Verlag, Berlin-New York, 1980.
\newblock Translated from the French by John Stillwell.

\bibitem[tD87]{Dieck}
Tammo tom Dieck.
\newblock {\em Transformation groups}, volume~8 of {\em De Gruyter Studies in
  Mathematics}.
\newblock Walter de Gruyter \& Co., Berlin, 1987.

\end{thebibliography}
\end{document}